\documentclass[11pt,reqno]{amsart}
\usepackage{amsfonts}
\usepackage{amsmath,amsthm,amssymb,url,listings,amsfonts,amscd}
\usepackage[alphabetic,initials]{amsrefs}
\usepackage{mathrsfs}
\usepackage{lipsum}
\usepackage{bbding}
\usepackage{graphicx,latexsym}
\usepackage{hyperref}
\usepackage{scrtime}
\usepackage{color}

\newcommand{\bea}{\begin{eqnarray}}
\newcommand{\eea}{\end{eqnarray}}
\newcommand{\bna}{\begin{eqnarray*}}
\newcommand{\ena}{\end{eqnarray*}}

\numberwithin{equation}{section}

\theoremstyle{plain}
\newtheorem{lemma}{Lemma}[section]
\newtheorem{theorem}[lemma]{Theorem}
\newtheorem{corollary}[lemma]{Corollary}
\newtheorem{proposition}[lemma]{Proposition}

\theoremstyle{definition}

\renewcommand{\Im}{\operatorname{Im}}

\begin{document}
	
\title{A Bessel $\delta$-method and hybrid bounds for $\mathrm{GL}_2$}

\author{Yilan Fan}
\date{}
\address{School of Mathematics and Statistics, Shandong University, Weihai\\Weihai, Shandong
        264209, China}
\email{YL.Fan@mail.sdu.edu.cn}
	
\author{Qingfeng Sun}
\date{}
\address{School of Mathematics and Statistics, Shandong University, Weihai\\Weihai, Shandong
		264209, China}
\email{qfsun@sdu.edu.cn}

\begin{abstract}
Let $g$ be a primitive holomorphic or Maass newform for $\Gamma_0(D)$.
In this paper, by studying the Bessel integrals associated to $g$,
we prove an asymptotic Bessel $\delta$-identity associated to $g$. Among other applications,
we prove the following hybrid subconvexity bound
\bna
L\left(1/2+it,g\otimes \chi\right)\ll_{g,\varepsilon}
(q(1+|t|))^{\varepsilon}q^{3/8}(1+|t|)^{1/3}
\ena
for any $\varepsilon>0$, where $\chi \bmod q$ is a primitive Dirichlet character
with $(q, D)=1$. This improves the previous known result.

\end{abstract}

\keywords{delta method, exponential sums,
subconvexity, hybrid bounds}
	
\subjclass[2010]{11F37, 11F66, 11L07, 11M41}

\maketitle
	
\section{Introduction}\label{introduction}

The circle method is a beautiful idea for investigating many problems
in analytic number theory. It originated in the study of the partition function
by Hardy and Ramanujan \cite{HR} in 1918. Let
\bna
\delta(n)=\left\{\begin{array}{ll}
1,&\mbox{if}\; n=0,\\
0,&\mbox{otherwise}
\end{array}
\right.
\ena
be the $\delta$-function. In the classical circle method one uses the equality
\bna
\delta(n)=\int_0^1e(\alpha n)\mathrm{d}\alpha
\ena
to pick out the interesting term from the associated generating function and
then decomposes the range of integration $[0,1]$ into Farey intervals around
rational approximations $a/c$ to $\alpha$ (see Vaughan \cite{V}).
In 1926,
Kloosterman \cite{K} made breakthrough by optimizing the subdivision structure properly
such that one may perform a nontrivial averaging over the numerators
of the approximating fractions $a/c$, which is usually called
{\it Kloosterman refinement}. Lately, Munshi \cite{munshi2} used
Kloosterman's version of the circle method to solve the $t$-aspect subconvexity
problem for $\mathrm{GL}_3$ $L$-functions.
A more recent variant of the circle method,
due to Jutila \cite{J1}, \cite{J2}, has optimized structure for subintervals
and flexible form for applications,
although its decomposition of the characteristic function of $[0,1]$
is approximate (see for example \cite{BHM}, \cite{munshi1} for some of its applications).

On the other hand, there are other variants of the circle method, which represent the
$\delta$-function as more tractable forms. An extremely useful one is the
$\delta$-method developed by Duke, Friedlander and Iwaniec \cite{DFI}, in which
the key formula
is a representation of the $\delta$-function in terms of the Ramanujan sums.
This was later revisited by Heath-Brown in the paper \cite{H}, where
an alternative formula for $\delta(n)$ was shown to be suitable for the study
of the classical circle method problems.
Browning and Vishe \cite{BV} adapted Heath-Brown's approach to the setting of arbitrary number
fields. For some applications
of the Duke-Friedlander-Iwaniec $\delta$-method, one may see
Munshi \cite{munshi3} and Pitt \cite{P}, for example.
More variants of the $\delta$-method or circle method can be found in
Munshi \cite{munshi4}-\cite{munshi7},
Aggarwal, Holowinsky, Lin and the second-named author \cite{AHLS},
and Aggarwal, Holowinsky, Lin and Qi \cite{AHLQ}.
These have promoted greatly the progress of the study on number theory problems.

A Bessel $\delta$-method
was first developed in Aggarwal, Holowinsky, Lin and Qi \cite{AHLQ}
for holomorphic cusp forms of $\mathrm{GL}_2$.
It is based on the following observation
\bna
&&\frac{1}{p}\sum_{a\bmod p}e\left(\frac{a(r-n)}{p}\right)
\int_0^{\infty}U\left(\frac{x}{X}\right)e\left(\frac{2\sqrt{rx}}{p}\right)
J_{\kappa-1}\left(\frac{4\pi \sqrt{nx}}{p}\right)\mathrm{d}x\\
&=&\boldsymbol{1}(r\equiv n\bmod p)\cdot \mathbf{1}(|r-n|<X^{\varepsilon}p\sqrt{N/X})
\cdot ``\text{some factor''}+``\text{error''}\\
&=&\delta(r-n)\cdot ``\text{some factor"}+``\text{error''},
\ena
provided that $r,n\asymp N$, $N<X^{1-\varepsilon}$ and $p^2<NX$, where $J_{\kappa-1}(x)$ denotes the stand $J$-Bessel function of order $\kappa-1$.
Here $\mathbf{1}(S)$ denotes the characteristic
function of $S$ defined by $\mathbf{1}(S)=1$ if $S$ is true, and equals 0 otherwise.
Notice that this Bessel integral arises naturally from
the Voronoi summation formula for holomorphic cusp forms
of $\mathrm{GL}_2$.
In this paper, we are concerned with the Bessel $\delta$-method for
Maass cusp forms of $\mathrm{GL}_2$.
The primary purpose of this paper is to extend
Aggarwal, Holowinsky, Lin and Qi's results to Maass forms. In fact,
we shall treat
both holomorphic cusp forms and Maass cusp forms. Among many other applications,
we prove a uniform subconvexity bound for associated twisted $L$-functions
in both $t$ and $q$ aspects.

Let $ S_{\kappa}^*(D,\xi_D)$ be the space of primitive holomorphic newforms of level $D$,
weight $\kappa$ and
nebentypus $\xi_D$ and $S_{\lambda}^*(D,\xi_D)$ be the space of primitive Maass newforms
of level $D$, weight zero, nebentypus $\xi_D$ and Laplace eigenvalue $\lambda=1/4+\mu^2$
for $\Gamma_0(D)$. For any $g\in S_{\kappa}^*(D,\xi_D)$ or $g\in S_{\lambda}^*(D,\xi_D)$, we define
\bea\label{I definition}
\mathbf{I}_g(a,b;X)=\int_0^{\infty}U\left(\frac{x}{X}\right)e\left(2a\sqrt{x}\right)
J_g\left(4\pi b\sqrt{x}\right)\mathrm{d}x,
\eea
where if $g$ is a holomorphic cusp form of weight $\kappa$,
\bea\label{J-1}
J_g(x)=2\pi i^{\kappa} J_{\kappa-1}(x),
\eea
and if $g$ is a Maass cusp form with Laplace eigenvalue $1/4+\mu^2$,
\bea\label{J-2}
J_g(x)=\frac{-\pi}{\sin (\pi i\mu)}\left(J_{2i \mu}(x)-J_{-2i \mu}(x)\right).
\eea
Here $J_{\nu}(x)$ ($\nu\in \mathbb{C}$) denotes the stand $J$-Bessel function,
and $U(x)\in C_c^{\infty}(1,2)$ is a nonnegative function.
The integral $\mathbf{I}_g(a,b;X)$ has the following properties.

\begin{proposition}\label{GL2-property}
Let $a,b\in \mathbb{R}^+$ and $b^2X\gg X^{\varepsilon}$. For any fixed integer $J\geq 0$, we have
\bna
\mathbf{I}_g(b,b;X)=C_U(b,X)\left(1+
O_{g,U,J}\left((b^2X)^{-(J+1)/2}\right)\right),
\ena
and
\bna
\mathbf{I}_g(a,b;X)\ll_{g,U,A}C_U(b,X)\left((a-b)^2X\right)^{-A}
\ena
for any $A\geq 0$, where
\bea\label{C-definition}
C_U(b,X)=X\sum_{j=0}^{J}\frac{d_j}{(4\pi b\sqrt{X})^{j+1/2}}
\widetilde{U}\left(\frac{3}{4}-\frac{j}{2}\right)\asymp \frac{X}{(b^2X)^{1/4}}
\eea
with $d_j$ ($j=0,1,2,\cdots,J$) being constants depending only on $g$
(If $g$ is holomorphic of weight $\kappa$, then $d_j=d_j(\kappa)$ depends
only on $\kappa$
and if $g$ is a Maass cusp form of spectral parameter $\mu$, then
$d_j=d_j(\mu)$ depends only on $\mu$).
In particular, $d_0=\pi^{1/2}(1-i)$. Here
$\widetilde{U}(s)=\int_0^{\infty}U(x)x^{s-1}\mathrm{d}x$ is the Mellin transform
of $U$. Moreover, $C_U(b,X)$ satisfies
\bea\label{C-U}
b^j\frac{\partial^j}{\partial a^j}C_U^{-1}(b,X)\ll_{g, U, J}  \frac{b^{1/2}}{X^{3/4}}.
\eea
\end{proposition}

To prove Proposition \ref{GL2-property},
we shall use asymptotic expansions of Bessel functions to treat
both holomorphic and Maass cusp forms.
As in \cite{AHLQ}, Proposition \ref{GL2-property} yields the following asymptotic
$\delta$-identity.
\begin{theorem}\label{thm: Bessel-delta}
Let  $p$ be prime and $N, X > 1$ be such that $ X^{1-\varepsilon} > \max\{N, p^2/N\}$.
Let $r, n$ be integers such that $r\asymp n\asymp N$. For any $A \geqslant 0$, we have
\bna
\delta(n-r) =\frac{1}{p} \sum_{ a \bmod p }
e \left( \frac {a(n-r)}{p} \right)  \cdot
\mathbf{I}_g\left(\frac{\sqrt{r}}{p}, \frac{\sqrt{n}}{p};X \right)C_U(\sqrt{r}/p,X)^{-1}
+ O_{g,U,A} \left(X^{-A}\right),
\ena
where $\mathbf{I}_g(a,b;X)$ and $C_U(a,X)$ are defined in \eqref{I definition}
and \eqref{C-definition}, respectively.
\end{theorem}

The main result proved in \cite{AHLQ} is the following theorem for
$g\in S_{\kappa}^*(D,\xi_D)$, which holds also for $g\in S_{\lambda}^*(D,\xi_D)$
by Theorem \ref{thm: Bessel-delta}. See \cite{AHLQ} for the proof.

\begin{theorem}\label{theorem-holo-Maass}
Let $g\in S_{\kappa}^*(D,\xi_D)$ or $S_{\lambda}^*(D,\xi_D)$ and denote by
$ \lambda_g(n)$ its $n$-th Fourier coefficient.
Let $V (x) \in C_c^{\infty} (0, \infty) $ be a smooth function
with support in $[1, 2]$. Assume that its total variation $\mathrm{Var} (V) \ll 1$
and that $V^{(j)} (x) \ll_{j} \Delta^j$  for $j \geqslant 0$ with $1\leq \Delta\leq TN^{-\varepsilon}$
for any $\varepsilon>0$.
For $\gamma$ real, and $ \phi (x) \in C^{\infty} (1/2, 5/2) $ satisfying
$ |\phi'' (x)| \gg 1 $ and  $  \phi^{(j)} (x) \ll_j 1 $ for $j \geqslant 1$,
define $f (x) = T \phi (x/ N) + \gamma x$.
Then
\bna
\sum_{n=1}^\infty \lambda_g(n) e \left(f(n)\right) V\left(\frac{n}{N}\right)
\ll_{g,\phi,\varepsilon}
T^{1/3} N^{1/2+\varepsilon} +  \frac{N^{1 +\varepsilon}}{T^{1/6}}.
\ena

\end{theorem}
	
To remove the smooth function in Theorem \ref{theorem-holo-Maass}, we assume that
$\lambda_g(n)\ll n^{\vartheta+\varepsilon}$ for any $\varepsilon>0$. Then we can take
$\vartheta=0$
for $g\in S_{\kappa}^*(D,\xi_D)$ by \cite{D}
and $\vartheta=7/64$ for $g\in S_{\lambda}^*(D,\xi_D)$ by \cite{Kim}.
Let $V\equiv 1$ on $[1+\Delta^{-1}, 2-\Delta^{-1}]$ with $\Delta=TN^{-\varepsilon}$.
Then one has the following estimate.

\begin{corollary}\label{main-corollary}
Same notations and assumptions as above. We have
\bna\label{unsmoothing}
\sum_{N \leqslant n \leqslant 2N}\lambda_g(n)\,e(f(n))
\ll_{g, \phi, \varepsilon}T^{1/3} N^{1/2+\varepsilon}+\frac{N^{1+\varepsilon}}{T^{1/6}}+
\frac{N^{1+\vartheta+\varepsilon}}{T}.
\ena
\end{corollary}

	A typical example is $\phi (x)=\pm x^{\beta}, \beta\neq 1$ and
$ f (x) = \alpha x^{\beta} + \gamma x$, $\alpha\neq 0$. The corresponding nonlinear exponential sum
\bna
\mathcal{S}_{\alpha, \beta, \gamma}^{\scriptscriptstyle \sharp} (N) =
\sum_{n \leqslant N}\lambda_g(n)\,e(\alpha n^{ \beta} + \gamma n)
\ena
has been studied by many authors (see \cite{J1}, \cite{AHLQ} and the references therein).
Applying Corollary \ref{main-corollary} with $T = |\alpha| N^{\beta}$, we get
\bea\label{bound-nonlinear sum}
\mathcal{S}_{\alpha, \beta, \gamma}^{\scriptscriptstyle \sharp} (N)
\ll_{g, \alpha, \beta,\varepsilon} N^{\frac{1}{2}+ \frac{\beta}{3}+\varepsilon} +
{N^{1-\frac{\beta}{6} + \varepsilon}}+N^{1-\beta+\vartheta+\varepsilon}.
\eea
Notice that the term $N^{1-\beta+\vartheta+\varepsilon}$ is dominated by the second term
if $\beta>6\vartheta/5$, which, in general, will not cause unnecessary problems in applications
This is easily seen for $g\in S_{\kappa}^*(D,\xi_D)$.
For $g\in S_{\lambda}^*(D,\xi_D)$, the second-named author and Wu \cite{SW} prove that
for $\beta\geq 1/2$ and $\gamma=0$,
\bna
\mathcal{S}_{\alpha, \beta, \gamma}^{\scriptscriptstyle \sharp} (N)
\ll_{g, \alpha, \beta,\varepsilon} N^{\beta+\varepsilon}.
\ena
Thus the first two terms in \eqref{bound-nonlinear sum} are better than the above bound in
the range $\beta>6/7$, in which case the condition $\beta>6\vartheta/5$ is met.

As in \cite{AHLQ}, another application of Theorem \ref{theorem-holo-Maass}
is concerned with the subconvexity problem of $L(s,g)$ in the $t$-aspect, where
\bna
L(s,g)=\sum_{n=1}^{\infty}\frac{\lambda_g(n)}{n^s}, \qquad \mathrm{Re}(s)>1.
\ena
The convex bound in the $t$-aspect is
$O\left((1+|t|)^{1/2+\varepsilon}\right)$
and for general $g$ the current record bound  is the so called Weyl-type bound,
i.e., $O\left((1+|t|)^{1/3+\varepsilon}\right)$. For $D=1$, the Weyl-type bound
has been proven by Good \cite{Good}, Jutila
\cite{J1}, \cite{J2} and Meurman \cite{Meu87}.
For general level, there are Weyl-type subconvexity results established by different
methods in
Booker, Milinovich and Ng \cite{BMN}, Aggarwal \cite{Agg}, and
Aggarwal, Holowinsky, Lin and Qi \cite{AHLQ}.
Applying Theorem \ref{theorem-holo-Maass}, we obtain the following Weyl-type bound.

\begin{theorem}\label{t-bound}
Let $g\in S_{\kappa}^*(D,\xi_D)$ or $S_{\lambda}^*(D,\xi_D)$.
For any $\varepsilon>0$, we have
\bna
L\left(1/2+it,g\right)\ll_{g,\varepsilon}
(1+|t|)^{1/3+\varepsilon}.
\ena
\end{theorem}

Now we proceed to describe another application of Proposition \ref{GL2-property}.
Let $g\in S_{\kappa}^*(D,\xi_D)$ or $S_{\lambda}^*(D,\xi_D)$.
Let $\chi$ be a primitive Dirichlet character modulo $q$ with $q$
prime and $(q,D)=1$.
The $L$-function associated to $g\otimes \chi$ is defined by
\bna
L(s,g\otimes \chi)=\sum_{n=1}^{\infty}\frac{\lambda_g(n)\chi(n)}{n^s}.
\ena
The convexity bound in the $t$ and $q$ aspects is
$
\left(q(1+|t|)\right)^{1/2+\varepsilon}.$
By looking through the proof, one finds that  \cite{AHLS} gives us
\bea\label{q-aspect}
L\left(1/2+it,g\otimes \chi\right)\ll_{g,\varepsilon}
(q(1+|t|))^{\varepsilon}(1+|t|)^{1/2}q^{3/8}
\eea
which is convexity in the $t$-aspect and of Burgess-type in the $q$
aspect,
and
\cite{AHLQ} yields
\bea\label{t-aspect}
L\left(1/2+it,g\otimes \chi\right)\ll_{g,\varepsilon}
(q(1+|t|))^{\varepsilon}q^{1/2}(1+|t|)^{1/3}
\eea
which is convexity in the $q$ aspect and of Weyl-type in the
$t$-aspect. We hope to improve upon the bounds in \eqref{q-aspect} and \eqref{t-aspect}.

Firstly, by combining Proposition \ref{GL2-property} with the trivial delta method in \cite{AHLS},
we give the following asymptotic Bessel $\delta$-identity.

\begin{theorem}\label{lem: Bessel-delta-2}
Let  $p$, $q$ be primes and $(p,q)=1$.
Let $M, X > 1$ be such that  $ X^{1-\varepsilon} > \max\left\{M,(pq)^2/ M\right\}$.
Let $m, n$ be integers such that $m\asymp n\asymp M$. For any $A \geqslant 0$, we have
\bea\label{delta-identity}
\delta(n-m)=\frac{1}{pq}\sum_{c|pq}\;
\sideset{}{^*}\sum_{a \bmod c}e\left(\frac{a(n-m)}{c}\right)
\mathbf{I}_g\left(\frac{\sqrt{m}}{pq}, \frac{\sqrt{n}}{pq}; X\right)
C_U\left(\frac{\sqrt{m}}{pq}, X\right)^{-1}
+O_{g,U,A}(X^{-A}),
\eea
where $\mathbf{I}_g(a,b;X)$ and $C_U(a,X)$ are defined in \eqref{I definition}
and \eqref{C-definition}, respectively. Here the $*$ in the sum over $a$ denotes the
sum is over $(a,c)=1$.
\end{theorem}
	
\begin{proof}
By the trivial delta method,
\bna
\delta(n-m)=\frac{1}{pq}\sum_{c|pq}\;
\sideset{}{^*}\sum_{a\bmod c}e\left(\frac{a(n-m)}{c}\right), \qquad \mbox{for} \;\;
pq>|n-m|.
\ena
On the other hand, by Proposition \ref{GL2-property}, for
$m\asymp n\asymp M$ and
$MX(pq)^{-2}\gg X^\varepsilon$,
$\mathbf{I}_g\left(\sqrt{m}/pq, \sqrt{n}/pq; X\right)$
is negligibly small unless
$$
\left(\frac{\sqrt{m}}{pq}- \frac{\sqrt{n}}{pq}\right)\sqrt{X}<X^\varepsilon,
\quad \mathrm{i.e.},\quad |n-m|<X^\varepsilon pq\sqrt{M/X}.
$$
Therefore, \eqref{delta-identity} follows immediately for
$M<X^{1-\varepsilon}.$
\end{proof}

As applications, we prove the following hybrid bound for
$L\left(1/2+it,g\otimes \chi\right)$ in both the $t$ and $q$ aspects.
\begin{theorem}\label{hybrid bound}
Let $g\in S_{\kappa}^*(D,\xi_D)$ or $S_{\lambda}^*(D,\xi_D)$ with $(q,D)=1$.
For any $\varepsilon>0$, we have
\bna
L\left(1/2+it,g\otimes \chi\right)\ll_{g,\varepsilon}
(q(1+|t|))^{\varepsilon}q^{3/8}(1+|t|)^{1/3}.
\ena
\end{theorem}

The hybrid subconvex bounds for $\mathrm{GL}_1$ were given by
Heath-Brown in \cite{H1} and \cite{H2}. For $\mathrm{GL}_2$
$L$-functions, this has been studied by Blomer and Harcos \cite{BH},
Michel and Venkatesh \cite{MV}, Munshi \cite{munshi0}, Wu \cite{Wu} and Kuan \cite{Kuan}.
For general $g$, the previous best known result is due to Wu \cite{Wu} who proved that
\bea\label{Wu's}
L\left(1/2+it,g\otimes \chi\right)\ll_{g,\varepsilon}
(q(1+|t|))^{\frac{1}{2}-\frac{1-2\vartheta}{8}+\varepsilon},
\eea
and recently for $g$ a primitive holomorphic cusp form of even weight $\kappa$, level $D$,
Kuan \cite{Kuan} proved that
\bea\label{Kuan's}
L\left(1/2+it,g\otimes \chi\right)\ll_{g,\varepsilon}
(q(1+|t|))^{\varepsilon}q^{\frac{3}{8}+\frac{\vartheta}{4}}(1+|t|)^{\frac{1}{3-2\vartheta}},
\eea
where we recall that $\vartheta$ is the value towards the Ramanujan-Petersson conjecture, so that
$\vartheta=0$
for $g\in S_{\kappa}^*(D,\xi_D)$ by \cite{D}
and $\vartheta=7/64$ for $g\in S_{\lambda}^*(D,\xi_D)$ by \cite{Kim}.
Notice that our Theorem \ref{hybrid bound} improve the bounds in \eqref{Wu's} and \eqref{Kuan's},
and does not depend on the Ramanujan-Petersson conjecture.

\section{Properties of Bessel integrals for $\mathrm{GL}_2$}

In this section, we prove Proposition \ref{GL2-property}.
Let $\mathbf{I}_g(a,b;X)$ be as in \eqref{I definition}.
First we consider the holomorphic case. By Section 7.21 in \cite{W}, for $x\gg 1$,we have
\bna
J_{\nu}(x)&=&\sqrt{\frac{2}{\pi x}}
\cos\left(x-\frac{\pi}{2}\nu-\frac{\pi}{4}\right)
\left\{\sum_{j=0}^{J-1}\frac{(-1)^j(\nu,2j)}{(2x)^{2j}}
+O_{n,\nu}\left(x^{-2J}\right)\right\}\nonumber\\
&&+\sqrt{\frac{2}{\pi x}}\sin\left(x-\frac{\pi}{2}\nu-\frac{\pi}{4}\right)
\left\{\sum_{j=0}^{J-1}\frac{(-1)^j(\nu,2j+1)}{(2x)^{2j+1}}
+O_{n,\nu}\left(x^{-2J-1}\right)\right\},
\ena
where $(\nu,0)=1$ and for $j\geq 1$,
\bna
(\nu,j)=\frac{\Gamma(\nu+j+1/2)}{j!\Gamma(\nu-j+1/2)}
=\frac{\{4\nu^2-1^2\}\{4\nu^2-3^2\}\cdots \{4\nu^2-(2j-1)^2\}}{2^{2j}j!}.
\ena
Thus for $x\gg 1$,
\bea\label{BoundOfJ}
 J_{\nu}(x)=x^{-1/2} \sum_{j=0}^{J}
 \frac{a_j e\left(x/2\pi\right)+b_j e\left(-x/2\pi\right)}{x^j}
 +O_{\nu,J}\left(x^{-3/2-J}\right),
\eea
where $a_j=a_j(\nu)$ and $b_j=b_j(\nu)$, $j=0,1,2,\cdots,J$ are constants depending only on $\nu$.
In particular, $a_0=(2\pi)^{-1/2}e(-(2\nu+1)/8)$ and
 $b_0=(2\pi)^{-1/2}e((2\nu+1)/8)$.
Thus
\bea\label{Bessel expansion: holomorphic}
 2\pi i^{\kappa}J_{\kappa-1}(4\pi b\sqrt{X}x)=\sum_{j=0}^{J}
 \frac{a_j' e\left( 2b\sqrt{X}x\right)
 +b_j' e\left(-2 b\sqrt{X}x\right)}{(4\pi b\sqrt{X}x)^{j+1/2}}
+O_{\kappa,J}\left((b\sqrt{X}x)^{-3/2-J}\right),
\eea
where $a_j'=a_j'(\kappa)$ and $b_j'=b_j'(\kappa)$,
 $j=0,1,2,\cdots,J$
are constants depending
only on $\kappa$.
In particular, $a_0'=\pi^{1/2}(1+i)$ and
$b_0'=\pi^{1/2}(1-i)$.

%If we take $J=0$, then
%\bea\label{Bessel expansion: holomorphic-0}
% 2\pi i^{\kappa}J_{\kappa-1}(4\pi b\sqrt{X}x)
% &=&\frac{1}{2}(b\sqrt{X}x)^{-1/2}\left(
% (1+i)e\left( 2b\sqrt{X}x\right)+
%(1-i)e\left(-2 b\sqrt{X}x\right)\right)\nonumber\\
% && \qquad\qquad\qquad\qquad\qquad\qquad+O_{\kappa}\left((b^2Xx^2)^{-3/4}\right).
% \eea

Next, we consider the case that $g$ is a Maass cusp form.
By \eqref{BoundOfJ}, for $x\gg 1$, we have
\bna
\frac{\pi\left(J_{2i\mu}(x)-J_{-2i\mu}(x)\right)}{{-\sin(\pi i\mu)}}
=x^{-1/2} \sum_{j=0}^{J}
 \frac{s_j e\left(x/2\pi\right)+t_j e\left(-x/2\pi\right)}{x^j}
 +O_{\mu,J}\left(x^{-3/2-J}\right),
\ena
where $s_j=s_j(\mu)$ and
$t_j=t_j(\mu)$, $j=0,1,2,\cdots,J$ are constants depending only on $\mu$.
In particular, $s_0=\pi^{1/2}(1+i)$ and $t_0=\pi^{1/2}(1-i)$. Thus
\bea\label{Bessel expansion: Maass}
&&\frac{\pi\left(J_{2i\mu}(4\pi b\sqrt{X}x)
- J_{-2i\mu}(4\pi b\sqrt{X}x)\right)}
{-\sin(\pi i\mu)}\nonumber\\
&=&
\sum_{j=0}^{J}\frac{s_j e\left( 2b\sqrt{X}x\right)
 +t_j e\left(-2 b\sqrt{X}x\right)}{(4\pi b\sqrt{X}x)^{j+1/2}}
 +O_{\mu,J}\left((b\sqrt{X}x)^{-3/2-J}\right).
\eea
%If we take $J=0$, then
%\bea\label{Bessel expansion: Maass-0}
% &&\frac{\pi\left(J_{2i\mu}(4\pi b\sqrt{X}x)
%- J_{-2i\mu}(4\pi b\sqrt{X}x)\right)}
%{-\sin(\pi i\mu)}\nonumber\\
%&=&\frac{1}{2}(b\sqrt{X}x)^{-1/2}\left(
% (1+i)e\left( 2b\sqrt{X}x\right)+
%(1-i)e\left(-2 b\sqrt{X}x\right)\right)+O_{\mu}\left((b^2Xx^2)^{-3/4}\right).
% \eea
By \eqref{Bessel expansion: holomorphic}and
 \eqref{Bessel expansion: Maass},
we have
\bea\label{J: expansion}
J_g(4\pi b\sqrt{X}x)=
\sum_{j=0}^{J}\frac{c_j e\left( 2b\sqrt{X}x\right)
 +d_j e\left(-2 b\sqrt{X}x\right)}{(4\pi b\sqrt{X}x)^{j+1/2}}+
 O_{g,J}\left((b\sqrt{X}x)^{-3/2-J}\right),
\eea
where $c_j$ and
$d_j$, $j=0,1,2,\cdots,J$ are constants depending only on $g$
(If $g$ is holomorphic of weight $\kappa$, these constants depend only on $\kappa$
and if $g$ is a Maass cusp form of spectral parameter $\mu$,
these constants depend only on $\mu$).
In particular, $c_0=\pi^{1/2}(1+i)$ and $d_0=\pi^{1/2}(1-i)$.

Now we turn to the integral in \eqref{I definition}.
Changing variable $x\rightarrow Xx^2$, one has
\bna
\mathbf{I}_g(a,b;X)
=2X\int_0^{\infty}xU\left(x^2\right)e\left(2a\sqrt{X}x\right)
J_g\left(4\pi b\sqrt{X}x\right)\mathrm{d}x.
\ena
Plugging \eqref{J: expansion} in and exchanging the orders of summation
and integration, we get
\bna
\mathbf{I}_g(a,b;X)
&=&2X\sum_{j=0}^{J}\frac{1}{(4\pi b\sqrt{X})^{j+1/2}}
\int_0^{\infty}x^{-j+1/2}U\left(x^2\right)
\\
&&\left(
 c_je\left( (2a+2b)\sqrt{X}x\right)+d_je\left( (2a-2b)\sqrt{X}x\right)
\right)\mathrm{d}x+O_{g,U,J}\left(\frac{X}{(b^2X)^{3/4+J/2}}\right).
\ena
Note that for $a,b\in \mathbb{R}^+$, by repeated partial integrations
\bna
\int_0^{\infty}x^{-j+1/2}U(x^2)e\left((2a+2b)\sqrt{X}x\right)\mathrm{dx}\ll_{A} ((a+b)^2X)^{-A}
\ena
for any $A>0$. Hence
\bea\label{general: a not equals b}
\mathbf{I}_g(a,b;X)
&=&2X\sum_{j=0}^{J}\frac{d_j}{(4\pi b\sqrt{X})^{j+1/2}}
\int_0^{\infty}x^{-j+1/2}U\left(x^2\right)
e\left( (2a-2b)\sqrt{X}x\right)\mathrm{d}x\nonumber\\
&&\qquad\qquad\qquad\qquad\qquad\qquad\qquad+O_{g,U,J}\left(\frac{X}{(b^2X)^{3/4+J/2}}\right).
\eea

If $a=b$, then
\bna
\mathbf{I}_g(b,b;X)
&=&2X\sum_{j=0}^{J}\frac{d_j}{(4\pi b\sqrt{X})^{j+1/2}}
\int_0^{\infty}x^{-j+1/2}U\left(x^2\right)\mathrm{d}x+O_{g,U,J}\left(\frac{X}{(b^2X)^{3/4+J/2}}\right)
\nonumber\\
&=&X\sum_{j=0}^{J}\frac{d_j}{(4\pi b\sqrt{X})^{j+1/2}}
\widetilde{U}\left(\frac{3}{4}-\frac{j}{2}\right)+
O_{g,U,J}\left(\frac{X}{(b^2X)^{3/4+J/2}}\right),
\ena
where  $\widetilde{U}(s)=\int_0^{\infty}U(x)x^{s-1}\mathrm{d}x$ is the Mellin transform
of $U$. Recall $d_0=\pi^{1/2}(1-i)$. Denote
\bea\label{C-definition-1}
C_U(b,X)=X\sum_{j=0}^{J}\frac{d_j}{(4\pi b\sqrt{X})^{j+1/2}}
\widetilde{U}\left(\frac{3}{4}-\frac{j}{2}\right)\asymp \frac{X}{(b^2X)^{1/4}}
\eea
Then $b^j\frac{\partial^j}{\partial a^j}C_U^{-1}(b,X)\ll b^{1/2}X^{-3/4}$ and
\bea\label{general: a equals b}
\mathbf{I}_g(b,b;X)
&=&C_U(b,X)\left(1+
O_{g,U,J}\left((b^2X)^{-\frac{J+1}{2}}\right)\right).
\eea

If $a\neq b$, then by repeated partial integrations, one has
\bea\label{integral: a not equals b}
\int_0^{\infty}x^{-j+1/2}U\left(x^2\right)
e\left( (2a-2b)\sqrt{X}x\right)\mathrm{d}x\ll_{A} ((a-b)^2X)^{-A}
\eea
for any $A \geq 0$.
Plugging \eqref{integral: a not equals b} into
\eqref{general: a not equals b}, one finds that
\bna
\mathbf{I}_g(a,b;X)\ll_{g,U,J,A} \frac{X}{(b^2X)^{1/4}((a-b)^2X)^{A}}+\frac{X}{(b^2X)^{3/4+J/2}}.
\ena
In view of \eqref{C-definition-1} and taking $J$ sufficiently large, we get
\bea\label{general: a not equals b-2}
\mathbf{I}_g(a,b;X)\ll_{g,U,A} C_U(b,X)((a-b)^2X)^{-A}.
\eea
Proposition \ref{GL2-property} follows from \eqref{C-definition-1}, \eqref{general: a equals b} and
\eqref{general: a not equals b-2}.

\section{A Hybrid bound for $\mathrm{GL_2}$ $L$-functions}

In this section, we prove Theorem \ref{hybrid bound}.
The proof is a combination of the methods in \cite{AHLS} and \cite{AHLQ}.
Without loss of generality, we assume $g$ is an even Maass cusp form and
suppose that $t>1$.
By the approximate functional equation of $L(s, g\otimes\chi)$, we have
\bna
L\left(\frac{1}{2}+it, g\otimes\chi\right)\ll
(qt)^\varepsilon \sup\limits_{N\leq (qt)^{1+\varepsilon}}
\frac{|\mathscr{S}(N)|}{\sqrt{N}}+O((qt)^{-2020}),
\ena
where
\bna
\mathscr{S}(N)=\sum_{n=1}^{\infty}\lambda_g(n)\chi(n)n^{-it}V
\left(\frac{n}{N}\right)
\ena
with $V\in C_c^{\infty}(1,2)$ satisfying $V^{(j)}(x)\ll_j 1$.
Estimating $\mathscr{S}(N)$ by Cauchy-Schwarz and the Rankin-Selberg estimate
\bea\label{Ranin-Selberg}
\sum_{n\leq N}|\lambda_g(n)|^2\ll_{g,\varepsilon} N^{1+\varepsilon},
\eea
one has the trivial bound $\mathscr{S}(N)\ll_{\varepsilon} N^{1+\varepsilon}$.
Thus
\bea\label{first estimate}
L\left(\frac{1}{2}+it, g\otimes\chi\right)
\ll (qt)^\varepsilon
\sup\limits_{q^{3/4}t^{2/3}<N\leq (qt)^{1+\varepsilon}}\frac{|\mathscr{S}(N)|}{\sqrt{N}}
+(qt)^{\varepsilon}q^{3/8}t^{1/3}.
\eea
Moreover, for $q^{3/4}t^{2/3}<N\leq (qt)^{1+\varepsilon}$, we assume
$q, t\gg N^{\varepsilon}$,
otherwise Theorem \ref{hybrid bound} follows from \eqref{q-aspect} and \eqref{t-aspect}.

Let $\mathcal{L}=\{\ell: \ell\in [L,2L], \ell\; \mbox{prime}\}$
with $q^{1/8}\leq L\leq q^{1/4}$ being a parameter to be chosen later.
Then $L\gg N^{\varepsilon}$.
Denote $L^{\star}=\sum\limits_{\ell\in \mathcal{L}}|\lambda_g(\ell)|^2$.
Then $L^{\star}\asymp L/\log L$.
By the Hecke relation
\bna
\lambda_g(\ell)\lambda_g(r)=\sum_{d|(\ell,n)}\lambda_g\left(\ell r/d^2\right)
=\lambda_g(\ell r)+\lambda_g(r/\ell)\mathbf{1}_{\ell|r},
\ena
and the Rankin-Selberg estimate
\eqref{Ranin-Selberg}, we get
\bea\label{S1-S2}
\mathscr{S}(N)=\mathscr{S}_1(N)+\mathscr{S}_2(N),
\eea
where
\bea\label{S1-expression}
\mathscr{S}_1(N)=\frac{1}{L^{\star}}\sum_{\ell\in \mathcal{L}}\overline{\lambda_g(\ell)}
\sum_{r=1}^{\infty}\chi(r)r^{-it}V\left(\frac{r}{N}\right)
\sum_{n}\lambda_g(n)\delta(n-r\ell)
\eea
and
\bna\label{S2-expression}
\mathscr{S}_2(N)=\frac{1}{L^{\star}}\sum_{\ell\in \mathcal{L}}\overline{\lambda_g(\ell)}
\chi(\ell)\ell^{-it}
\sum_{n=1}^{\infty}\lambda_g(n)\chi(n)n^{-it}V\left(\frac{\ell n}{N}\right)
\ll \sup_{\ell\asymp L}\left|\mathscr{S}\left(\frac{N}{\ell}\right)\right|.
\ena
Here the last inequality follows from Cauchy-Schwarz inequality.
By repeating the above reduction process for $\mathscr{S}\left(N/\ell\right)$, we
get
\bea\label{S2-estimate}
\mathscr{S}_2(N)
&\ll& \sum_{j=0}^{J-1}\sup_{\ell_i\asymp L\atop i=0,1,\ldots j}
\left|\mathscr{S}_1\left(\frac{N}{\ell_0\ell_1\cdots \ell_j}\right)\right|
+ \sup_{\ell_i\asymp L\atop i=0,1,\ldots J}
\left|\mathscr{S}\left(\frac{N}{\ell_0\ell_1\cdots \ell_J}\right)\right|,
\eea
where by the support of $V$, the last term vanishes if we take $J$ such that
$L^{J+1}\gg N^{(J+1)\varepsilon}\gg N$, i.e., $J\gg [\varepsilon^{-1}]$.
By \eqref{S1-S2} and \eqref{S2-estimate}, we obtain
\bea\label{reduction}
\mathscr{S}(N)=\mathscr{S}_1(N)+
\sum_{j=0}^{[\varepsilon^{-1}]-1}\sup_{\ell_i\asymp L\atop i=0,1,\ldots j}
\left|\mathscr{S}_1\left(\frac{N}{\ell_0\ell_1\cdots \ell_j}\right)\right|.
\eea
Therefore, we only need to estimate $\mathscr{S}_1(N)$ in \eqref{S1-expression}, since the other terms
are smaller and
can be estimated similarly.

Let $\mathcal{P}=\{p: p\in [P,2P], p\; \mbox{prime}, (p, qD)=1\}$
with $P$ being a parameter to be chosen later.
Denote $P^{\star}=\sum\limits_{p\in \mathcal{P}}1\asymp P/\log P$.
Assume
\bea\label{assumption 1}
X^{1-\varepsilon} > \max\left\{NL,(Pq)^2/(NL)\right\}.
\eea
Summing the $\delta$-identity in \eqref{delta-identity}
over $p\in \mathcal{P}$, we get
\bea\label{delta-expression}
\delta(n-r\ell)&=&\frac{1}{P^{\star}}\sum_{p\in\mathcal{P}}
\frac{1}{pq}\sum_{c|pq}\;
\sideset{}{^*}\sum_{a \bmod c}e\left(\frac{a(n-r\ell)}{c}\right)\nonumber\\
&&\times C_U\left(\frac{\sqrt{r\ell}}{pq}, X\right)^{-1}
\mathbf{I}_g\left(\frac{\sqrt{r\ell}}{pq}, \frac{\sqrt{n}}{pq}; X\right)+O_{g,U,A}(X^{-A}).
\eea
Plugging \eqref{delta-expression} into \eqref{S1-expression}, one has
\bea\label{S1-expression-2}
\mathscr{S}_1(N)&=&\frac{1}{qL^{\star}P^{\star}}\sum_{\ell\in \mathcal{L}}
\overline{\lambda_g(\ell)}\sum_{p\in\mathcal{P}}
\frac{1}{p}\sum_{c|pq}\;
\sideset{}{^*}\sum_{a \bmod c}\;
\sum_{r=1}^{\infty}\chi(r)e\left(-\frac{ar\ell}{c}\right)r^{-it}
\frac{(NL)^{1/4}}{X^{3/4}(Pq)^{1/2}}
\widetilde{V}_{p,\ell}\left(\frac{r}{N}\right)\nonumber\\
&&\sum_{n}\lambda_g(n)e\left(\frac{an}{c}\right)
\mathbf{I}_g\left(\frac{\sqrt{r\ell}}{pq}, \frac{\sqrt{n}}{pq}; X\right)
,
\eea
where
$
\widetilde{V}_{p,\ell}\left(r/N\right)=(N L)^{-1/4}X^{3/4}
(Pq)^{1/2}C_U^{-1}\left(\sqrt{r\ell}/(pq), X\right)V\left(r/N\right)
$
satisfying $\widetilde{V}_{p,\ell}^{(j)}\left(y\right)\ll_j 1$ by \eqref{C-U}.

We recall the following Voronoi formula for $\mathrm{GL}_2$ (see \cite[Theorem A.4]{KMV}).

\begin{lemma}\label{voronoiGL2}
Let $g\in S_{\kappa}^*(D,\xi_D)$ (resp. $S_{\lambda}^*(D,\xi_D)$).
Let $c\in \mathbb{N}$ and $a, \overline{a}\in \mathbb{Z}$ be such
that $(a,c)=1$, $a\overline{a}\equiv 1(\bmod \;c)$ and $(c,D)=1$.
Let $F\in C_c^{\infty}(\mathbb{R}^+)$. Then there exists a complex number
$\eta_g(D)$ of modulus 1 and a newform $g^*\in S_{\kappa}^*(D,\overline{\xi_D})$
(resp. $S_{\lambda}^*(D,\overline{\xi_D})$) such that
\bna
&&\sum_{n=1}^{\infty}\lambda_g(n)e\left(\frac{a n}{c}\right)F(n)\\
&=&\xi_{D}(-c)\frac{\eta_g(D)}{c\sqrt{D}}
\sum_{n=1}^{\infty}\lambda_{g^*}(n)e\left(-\frac{\overline{aD}n}{c}\right)
\int_0^{\infty}F(x)J_g\left(\frac{4\pi}{c}\sqrt{\frac{nx}{D}}\right)
\mathrm{d}x\\
&&+\xi_{D}(c)\frac{\eta_g(D)}{c\sqrt{D}}
\sum_{n=1}^{\infty}\lambda_{g^*}(n)e\left(\frac{\overline{a D}n}{c}\right)
\int_0^{\infty}F(x)K_g\left(\frac{4\pi}{c}\sqrt{\frac{nx}{D}}\right)
\mathrm{d}x,
\ena
where $J_g(x)$ is defined as in \eqref{J-1} and \eqref{J-2}, if $g\in S_{\kappa}^*(D,\xi_D)$,
$K_g(x)=0$, and if $g\in S_{\lambda}^*(D,\xi_D)$,
\begin{equation}\label{K expression}
K_g(x)= 4\varepsilon_g\cosh(\pi \mu)K_{2i\mu}(x).
\end{equation}
Here $\varepsilon_g$ is an eigenvalue of $g$ under the reflection operator.
\end{lemma}

By Lemma \ref{voronoiGL2}, we have
\bna
&&\sum_{n=1}^{\infty}\lambda_{g}(n)e\left(\frac{an}{c}\right)
\int_0^{\infty}F\left(\frac{Dc^2x}{p^2q^2}\right)
J_g\left(\frac{4\pi\sqrt{nx}}{pq}\right)
\mathrm{d}x\\
&=&\frac{\overline{\xi_{D}}(-c)}{\eta_g(D)D^{1/2}}\frac{p^2q^2}{c}
\sum_{n=1}^{\infty}\lambda_{g^*}(n)e\left(-\frac{\overline{aD}n}{c}\right)F(n)\\
&&-\overline{\xi_{D}}(-1)
\sum_{n=1}^{\infty}\lambda_{g}(n)e\left(-\frac{an}{c}\right)
\int_0^{\infty}F\left(\frac{Dc^2x}{p^2q^2}\right)
K_g\left(\frac{4\pi\sqrt{nx}}{pq}\right)
\mathrm{d}x.
\ena
Applying this relation with
$$
F(x)=U\left(\frac{p^2q^2x}{DXc^2}\right)
e\left(\frac{2}{c}\sqrt{\frac{r\ell x}{D}}\right),
$$
we write the $n$-sum in \eqref{S1-expression-2} as
\bea\label{relation}
&&\sum_{n}\lambda_g(n)e\left(\frac{an}{c}\right)
\mathbf{I}_g\left(\frac{\sqrt{r\ell}}{pq}, \frac{\sqrt{n}}{pq}; X\right)\nonumber\\
&=&\frac{\overline{\xi_{D}}(-c)}{\eta_g(D)D^{1/2}}\frac{p^2q^2}{c}
\sum_{n=1}^{\infty}\lambda_{g^*}(n)e\left(-\frac{\overline{aD}n}{c}\right)
U\left(\frac{p^2q^2n}{DXc^2}\right)
e\left(\frac{2}{c}\sqrt{\frac{r\ell n}{D}}\right)\nonumber\\
&&-\overline{\xi_{D}}(-1)
\sum_{n=1}^{\infty}\lambda_{g}(n)e\left(-\frac{an}{c}\right)
\int_0^{\infty}U\left(\frac{x}{X}\right)
e\left(\frac{2\sqrt{r\ell x}}{pq}\right)
K_g\left(\frac{4\pi\sqrt{nx}}{pq}\right)
\mathrm{d}x.
\eea
Here we recall that by \eqref{I definition},
\bna
\mathbf{I}_g\left(\frac{\sqrt{r\ell}}{pq}, \frac{\sqrt{n}}{pq}; X\right)
=\int_0^{\infty}U\left(\frac{x}{X}\right)
e\left(\frac{2\sqrt{r\ell x}}{pq}\right)
J_g\left(\frac{4\pi\sqrt{nx}}{pq}\right)\mathrm{d}x.
\ena
Notice that $\mu\in \mathbb{R}^+$ or $i\mu\in [-1/4,1/4]$.
For $z\ll 1$, the Taylor expansion of $K_v(z)$ (see (8.403-1), (8.405-1),
(8.407-1), (8.402) and (8.447) in \cite{GR}) yields
\bna
z^jK_{2i\mu}^{(j)}(z)\ll z^{2|\Im \mu|}.
\ena
This hold also for $z\gg 1$, since
\bna
K_v(z)=\sqrt{\frac{2}{\pi z}}e^{-z}(1+O_v(z^{-1})), \qquad \mbox{for}\; z\gg 1
\ena
(see \cite[(8.451-6)]{GR}).
By repeated partial integrations and \eqref{K expression}, one has
\bna
&&\int_0^{\infty}U\left(\frac{x}{X}\right)
e\left(\frac{2\sqrt{r\ell x}}{pq}\right)
K_g\left(\frac{4\pi\sqrt{nx}}{pq}\right)
\mathrm{d}x\\
%&=&8\varepsilon_gX\cosh(\pi \mu)\int_0^{\infty}xU\left(x^2\right)
%e\left(\frac{2\sqrt{r\ell X}x}{pq}\right)
%K_{2i\mu}\left(\frac{4\pi\sqrt{nX}x}{pq}\right)
%\mathrm{d}x\\
&=&8\varepsilon_gX\cosh(\pi \mu)\left(-\frac{4\pi i\sqrt{r\ell X}}{pq}\right)^{-j}\int_0^{\infty}\left(xU\left(x^2\right)
K_{2i\mu}\left(\frac{4\pi\sqrt{nX}x}{pq}\right)\right)^{(j)}
e\left(\frac{2\sqrt{r\ell X}x}{pq}\right)
\mathrm{d}x\\
&\ll&X\left(\frac{\sqrt{nX}}{pq}\right)^{2|\Im \mu|}\left(\frac{\sqrt{r\ell X}}{pq}\right)^{-j}
\ena
for any $j\geq 0$. Since $r\asymp N$, $\ell\asymp L$, $p\asymp P$ and
$\sqrt{NLX}/Pq\gg X^{\varepsilon}$ by \eqref{assumption 1}, the above
inequality implies that the second term on the right hand side of \eqref{relation}
is negligible.

Plugging \eqref{relation} into \eqref{S1-expression-2}, we arrive that
\bea\label{S1-expression-3}
\mathscr{S}_1(N)&=&\frac{1}{\eta_g(D)D^{1/2}}\frac{(NL)^{1/4}
q^{1/2}}{X^{3/4}P^{1/2}L^{\star}P^{\star}}\sum_{\ell\in \mathcal{L}}
\overline{\lambda_g(\ell)}\sum_{p\in\mathcal{P}}p
\sum_{c|pq}\frac{\overline{\xi_{D}}(-c)}{c}
\sum_{n=1}^{\infty}\lambda_{g^*}(n)
U\left(\frac{p^2q^2n}{DXc^2}\right)
\nonumber\\
&&\times
\sum_{r=1}^{\infty}\chi(r)S(\overline{D}n, r\ell; c)
r^{-it}
\widetilde{V}_{p,\ell}\left(\frac{r}{N}\right)e\left(\frac{2}{c}\sqrt{\frac{r\ell n}{D}}\right).
\eea

To simplify our analysis, we let
\bea\label{assumption 2}
X=\frac{{q}^2K^2P^2}{NL}, \qquad N^\varepsilon<t^{1/2}<K<t^{1-\varepsilon}.
\eea
Then the assumption $NL<X^{1-\varepsilon}$ in \eqref{assumption 1} amounts to
\bea\label{assumption 3}
P>N^{1+\varepsilon}L/(qK).
\eea

\subsection{First application of Poisson summation}

As in \cite{AHLS}, we denote $a_b=a/(a,b)$,
where $(a,b)$ is the gcd of $a$ and $b$, and
$[a,b]$ denotes the lcm of $a$ and $b$.
Breaking the sum over $r$ modulo $[c,q]$ in \eqref{S1-expression-3} and applying the Poisson summation formula,
one has
\bna
&&\sum_{r=1}^{\infty}\chi(r)S(\overline{D}n, r\ell; c)
r^{-it}
\widetilde{V}_{p,\ell}\left(\frac{r}{N}\right)e\left(\frac{2}{c}\sqrt{\frac{r\ell n}{D}}\right)\\
&=&
\sum_{\beta (\bmod [c,q])}\chi(\beta)S(\overline{D}n,\beta\ell; c)
\sum_{r\equiv\beta(\bmod [c,q])}r^{-it}
\widetilde{V}_{p,\ell}\left(\frac{r}{N}\right)
e\left(\frac{2}{c}\sqrt{\frac{r\ell n}{D}}\right)\\
&=&
\frac{N^{1-it}}{[c,q]}\sum_{r\in\mathbb{Z}}
\left(\sum_{\beta(\bmod [c,q])}\chi(\beta)S(\overline{D}n,\beta\ell; c)e\left(\frac{\beta r}{[c,q]}\right)\right)
\mathcal{J}(n,r,\ell;c,p),
\ena
where
\bea\label{integral-J}
\mathcal{J}(n,r,\ell;c,p)=\int_0^\infty\widetilde{V}_{p,\ell}(y)
e\left(-\frac{t}{2\pi}\log y+\frac{2}{c}\sqrt{\frac{N\ell ny}{D}}
-\frac{rNy}{[c,q]}\right)\mathrm{d}y.
\eea

Using the relation $[c,q]=qc_{q}$ and reciprocity, the $\beta$-sum can be written as
\bna
&&\sideset{}{^*}\sum_{a \bmod c}
e\left(\frac{-\overline{Da}n}{c}\right)
\sum_{\beta(\bmod [c,q])}\chi(\beta)
e\left(\frac{-a\beta\ell}{c}\right)
e\left(\frac{\beta r}{[c,q]}\right)\nonumber\\
&=&
\sideset{}{^*}\sum_{a \bmod c}
e\left(\frac{-\overline{Da}n}{c}\right)
\sum_{\beta(\bmod q)}\chi(\beta)
e\left(\frac{(r-a\ell q_c)\overline{c_q}\beta}{q}\right)
\sum_{\beta(\bmod c_q)}
e\left(\frac{(r-a\ell q_c)\overline{q}\beta}{c_q}\right)\nonumber\\
&=&
c_{q}g_\chi\sideset{}{^*}
\sum_{a \bmod c \atop r\equiv a\ell q_c\bmod c_{q}}
\overline{\chi}\left((r-a\ell q_c)\overline{c_{q}}\right)
e\left(\frac{-\overline{Da}n}{c}\right),
\ena
where $g_\chi$ is the Gauss sum.

To estimate the integral in \eqref{integral-J}, we denote its phase function by
$\rho(y)$. Then
\bna
\rho'(y)=-\frac{t}{2\pi y}+\frac{1}{c}\sqrt{\frac{N\ell n}{Dy}}-\frac{rN}{[c,q]}
\ena
and by \eqref{assumption 2},
\bna
\rho''(y)=\frac{t}{2\pi y^2}-\frac{1}{2c}\sqrt{\frac{N\ell n}{Dy^{3/2}}}
\asymp \max\left\{t, \frac{\sqrt{NLX}}{pq}\right\}=\max\{t,K\}=t.
\ena
By Lemma \ref{lem: upper bound}, $\mathcal{J}(n,r,\ell;c,p)$ is negligibly small unless
$$
\frac{|r|N}{[c,q]}\ll N^\varepsilon t \qquad \mbox{for\, any }\,\varepsilon>0.
$$
Accordingly, we can effectively truncate the sum at $|r|\ll [c,q]t/N^{1-\varepsilon}$,
at the cost of a negligibly error.
For smaller $r$, by the second derivative test in \eqref{lem: derivative tests, dim 1}, one has
\bea\label{second derivative bound}
\mathcal{J}(n,r,\ell;c,p)\ll t^{-1/2}.
\eea
Consequently, $\mathscr{S}_1(N)$ in \eqref{S1-expression-3} is transformed into
\bna
\mathscr{S}_1(N)&=&\frac{1}{\eta_g(D)D^{1/2}}
\frac{N^{2-it}Lg_\chi}{q^2P^2K^{3/2}L^{\star}P^{\star}}\sum_{\ell\in \mathcal{L}}
\overline{\lambda_g(\ell)}\sum_{p\in\mathcal{P}}p
\sum_{c|pq}\frac{\overline{\xi_{D}}(-c)}{c}
\sum_{n=1}^{\infty}\lambda_{g^*}(n)
U\left(\frac{p^2q^2n}{DXc^2}\right)\nonumber\\&&\times
\sum_{|r|\ll\frac{[c,q]t}{N^{1-\varepsilon}}}\;
\sideset{}{^*}
\sum_{a \bmod c \atop r\equiv a\ell q_c\bmod c_{q}}
\overline{\chi}\left((r-a\ell q_c)\overline{c_{q}}\right)
e\left(\frac{-\overline{Da}n}{c}\right)
\mathcal{J}(n,r,\ell;c,p).
\ena
When $c=1$, by \eqref{Ranin-Selberg} and \eqref{second derivative bound},
its contribution to $\mathscr{S}_1(N)$ above is bounded by
\bna
&&N^\varepsilon
\frac{N^2q^{1/2}}{{q}^2P^2K^{3/2}}
\sum_{\ell\in \mathcal{L}}|\lambda_g(\ell)|
\sum_{p\in\mathcal{P}}
\sum_{n\asymp DX/(pq)^2}|\lambda_{g^*}(n)|
\sum_{|r|\ll\frac{qt}{N^{1-\varepsilon}}}
|\mathcal{J}(n,r,\ell;1,p)|\nonumber\\
&\ll_{g,\varepsilon}&
\frac{N^{2+\varepsilon}}{{q}^{3/2}P^{5/2}K^{3/2}}
LP\frac{X}{p^2q^2}
\frac{qt}{N}\frac{1}{\sqrt{t}}\nonumber\\
&\ll_{g,\varepsilon}&
\frac{t^{1/2}K^{1/2}}{q^{1/2}PN^{1-\varepsilon}},
\ena
recalling $X={q}^2K^2P^2/(NL)$ in \eqref{assumption 2}.
%Similarly, the contribution from $c=p$ to $\mathscr{S}_1(N)$ in \eqref{S1-expression-4} is
%at most
%\bea\label{c=p}
%&&N^\varepsilon
%\frac{N^2q^{1/2}}{{q}^2P^{5/2}K^{3/2}}
%\sum_{\ell\in \mathcal{L}}
%|\lambda_g(\ell)|\sum_{p\in\mathcal{P}}
%p^{-1/2}
%\sum_{n\asymp DX/q^2}|\lambda_{g^*}(n)|
%\sum_{|r|\ll\frac{pqt}{N^{1-\varepsilon}}\atop (r,pq)=1}
%\left|\mathcal{J}(n,r,\ell;p,p)\right|\nonumber\\
%&\ll&
%N^\varepsilon
%\frac{N^2}{q^{3/2}P^{5/2}K^{3/2}}
%LP^{1/2}\frac{X}{q^2}\frac{pqt}{N}\frac{1}{\sqrt{t}}\nonumber\\
%&\ll&
%N^\varepsilon
%\frac{Pt^{1/2}K^{1/2}}{q^{1/2}}.
%\eea
Thus
\bea\label{S1-expression-4}
\mathscr{S}_1(N)=\mathscr{S}^{\sharp}(N)+\mathscr{S}_{c=p}(N)
+\mathscr{S}_{c=q}(N)
+O_{g,\varepsilon}\left(\frac{t^{1/2}K^{1/2}}{q^{1/2}PN^{1-\varepsilon}}\right),
\eea
where
\bea\label{S-sharp}
\mathscr{S}^{\sharp}(N)&=&\frac{\overline{\xi_{D}}(-q)}{\eta_g(D)D^{1/2}}
\frac{N^{2-it}Lg_\chi}{{q}^3P^2K^{3/2}P^{\star}L^{\star}}
\sum_{\ell\in \mathcal{L}}\overline{\lambda_g(\ell)}\sum_{p\in\mathcal{P}}
\chi(p)\overline{\xi_{D}}(p)\sum_{n=1}^\infty\lambda_{g^*}(n)
U\left(\frac{n}{DX}\right)\nonumber\\
&&
\sum_{0\neq |r|\ll\frac{pqt}{N^{1-\varepsilon}} \atop (r,p)=1}
\mathcal{J}(n,r,\ell;pq,p)
\;\sideset{}{^*}\sum_{a \bmod pq \atop r\equiv a\ell \bmod p}
\overline{\chi}(r-a\ell)e\left(\frac{-\overline{Da}n}{pq}\right),
\eea
\bea\label{S-c-p}
\mathscr{S}_{c=p}(N)&=&\frac{1}{\eta_g(D)D^{1/2}}
\frac{N^{2-it}Lg_\chi}{q^2P^2K^{3/2}L^{\star}P^{\star}}\sum_{\ell\in \mathcal{L}}
\overline{\lambda_g(\ell)}\sum_{p\in\mathcal{P}}
\overline{\xi_{D}}(-p)\chi(p)
\sum_{n=1}^{\infty}\lambda_{g^*}(n)
U\left(\frac{q^2n}{DX}\right)\nonumber\\&&\times
\sum_{0\neq  |r|\ll\frac{pqt}{N^{1-\varepsilon}}\atop (r,p)=1}\overline{\chi}(r)
e\left(\frac{-\overline{Dr}\ell q n}{p}\right)
\mathcal{J}(n,r,\ell;p,p)
\eea
and
\bea\label{S-c-q}
\mathscr{S}_{c=q}(N)&=&\frac{\overline{\xi_{D}}(-q)}{\eta_g(D)D^{1/2}}
\frac{N^{2-it}Lg_\chi}{q^3P^2K^{3/2}L^{\star}P^{\star}}\sum_{\ell\in \mathcal{L}}
\overline{\lambda_g(\ell)}\sum_{p\in\mathcal{P}}p
\sum_{n=1}^{\infty}\lambda_{g^*}(n)
U\left(\frac{p^2n}{DX}\right)\nonumber\\&&\times
\sum_{|r|\ll\frac{qt}{N^{1-\varepsilon}}}\mathcal{J}(n,r,\ell;q,p)
\sideset{}{^*}
\sum_{a \bmod q}
\overline{\chi}\left(r-a\ell\right)
e\left(\frac{-\overline{Da}n}{q}\right),
\eea
where $\mathcal{J}(n,r,\ell;c,p)$ is defined in \eqref{integral-J}.

We will first deal with the most complicated term $\mathscr{S}^{\sharp}(N)$, and
leave $\mathscr{S}_{c=q}(N)$ and $\mathscr{S}_{c=p}(N)$
to the Sections 3.4 and 3.6. Since $(p,q)=1$, the sum over $a$ in \eqref{S-sharp} factors as
\bna
&&\sideset{}{^*}\sum_{a \bmod q}
\overline{\chi}(r-a\ell)e\left(-\frac{\overline{Dpa}n}{q}\right)
\sideset{}{^*}\sum_{b \bmod p \atop b\equiv\overline{\ell}r \bmod p}
e\left(-\frac{\overline{Dqb}n}{p}\right)\\
&=&e\left(-\frac{\overline{Dqr}\ell n}{p}\right)
\sideset{}{^*}\sum_{a \bmod q}
\overline{\chi}(r-a)e\left(-\frac{\overline{Dpa}\ell n}{q}\right).
\ena
Thus
\bea\label{S-sharp-1-1}
\mathscr{S}^{\sharp}(N)&=&\frac{\overline{\xi_{D}}(-q)}{\eta_g(D)D^{1/2}}
\frac{N^{2-it}Lg_\chi}{{q}^3P^2K^{3/2}P^{\star}L^{\star}}
\sum_{n=1}^\infty\lambda_{g^*}(n)
U\left(\frac{n}{DX}\right)
\sum_{\ell\in \mathcal{L}}\overline{\lambda_g(\ell)}\sum_{p\in\mathcal{P}}
\chi(p)\overline{\xi_{D}}(p)\nonumber\\
&&
\sum_{0\neq |r|\ll\frac{Pqt}{N^{1-\varepsilon}} \atop (r,p)=1}
e\left(-\frac{\overline{Dqr}\ell n}{p}\right)
\mathcal{J}(n,r,\ell;pq,p)
\sideset{}{^*}\sum_{a \bmod q}
\overline{\chi}(r+a)e\left(\frac{\overline{Dpa}\ell n}{q}\right).
\eea

\subsection{Cauchy-Schwarz and Poisson summation}
Denote
\bea\label{R definition}
R=\frac{Pqt}{N^{1-\varepsilon}}.
\eea
Recalling \eqref{assumption 2},
by applying Cauchy-Schwarz to \eqref{S-sharp-1-1}
and using \eqref{Ranin-Selberg}, we get
\bea\label{S-sharp-inequality}
\mathscr{S}^{\sharp}(N)
\ll_{g,\varepsilon}\frac{N^{3/2+\varepsilon}}{{q}^{3/2}P^2K^{1/2}L^{1/2}}
\mathbf{T}^{1/2},
\eea
where temporarily,
\bna
\mathbf{T}&=&
\sum_{n\in\mathbb{Z}}U\left(\frac{n}{DX}\right)
\left|\sum_{\ell\in \mathcal{L}}\overline{\lambda_g(\ell)}
\sum_{p\in\mathcal{P}}\chi(p)\overline{\xi_{D}}(p)
\sum_{0\neq |r|\ll R \atop (r,p)=1}
e\left(-\frac{\overline{Dqr}\ell n}{p}\right)
\right.\\&&\qquad\qquad\qquad\times\left.\mathcal{J}(n,r,\ell;pq,p)
\sideset{}{^*}\sum_{a \bmod q}
\overline{\chi}(r+a)e\left(\frac{\overline{Dpa}\ell n}{q}\right)
\right|^2.
\ena
Opening the square and switching the order of summations,
\bna
\mathbf{T}&=&
\sum_{\ell_1\in \mathcal{L}}\overline{\lambda_g(\ell_1)}
\sum_{\ell_2\in \mathcal{L}}\lambda_g(\ell_2)
\sum_{p_1\in\mathcal{P}}\chi(p_1)\overline{\xi_{D}}(p_1)
\sum_{p_2\in\mathcal{P}}\overline{\chi}(p_2)\xi_{D}(p_2)
\sum_{0\neq |r_1|\ll R \atop (r_1,p_1)=1}
\sum_{0\neq |r_2|\ll R \atop (r_2,p_2)=1}\\
&&
\sideset{}{^*}\sum_{a_1 \bmod q}\overline{\chi}(r_1+a_1)
\sideset{}{^*}\sum_{a_2 \bmod q}\chi(r_2+a_2)
\times\widetilde{\mathbf{T}},
\ena
where
\bna
\widetilde{\mathbf{T}}&=&\sum_{n\in\mathbb{Z}} U\left(\frac{n}{DX}\right)
e\left(-\frac{\overline{Dqr_1}\ell_1 n}{p_1}\right)
e\left(\frac{\overline{Dqr_2}\ell_2 n}{p_2}\right)\\&&\times
e\left(\frac{\overline{Dp_1a_1}\ell_1 n}{q}\right)
e\left(-\frac{\overline{Dp_2a_2}\ell_2 n}{q}\right)
\mathcal{J}(n,r_1,\ell_1;p_1q,p_1)
\overline{\mathcal{J}(n,r_2,\ell_2;p_2q,p_2)}.
\ena
Breaking the $n$-sum in $\widetilde{\mathbf{T}}$ into residue classes
modulo $p_1p_2q$ and applying Poisson summation, we get
\bna
\widetilde{\mathbf{T}}&=&\frac{DX}{p_1p_2q}\sum_{n\in\mathbb{Z}}\;\sum_{b \bmod p_1p_2q}
e\left(-\frac{\overline{Dqr_1}\ell_1 b}{p_1}
+\frac{\overline{Dqr_2}\ell_2 b}{p_2}
+\frac{\overline{Dp_1a_1}\ell_1 b}{q}
-\frac{\overline{Dp_2a_2}\ell_2 b}{q}
\right)\nonumber\\&&\qquad\qquad\qquad\qquad\qquad\qquad\qquad\qquad\qquad\qquad\times
e\left(\frac{bn}{p_1p_2q}\right)
\mathfrak{J}\left(\frac{D X n}{p_1p_2q}\right),
\ena
where
\bea\label{J-definition}
\mathfrak{J}(x)=\int_{\mathbb{R}}U(y)\mathcal{J}(DXy,r_1,\ell_1;p_1q,p_1)
\overline{\mathcal{J}(DXy,r_2,\ell_2;p_2q,p_2)}e(-xy)\mathrm{d}y.
\eea
Since $\left(p_1p_2,q\right)=1$, we apply reciprocity to write
\bna
\widetilde{\mathbf{T}}&=&\frac{DX}{p_1p_2q}\sum_n\;
\sum_{b_1 \bmod p_1p_2}
e\left(\frac{(-\overline{Dr_1}\ell_1p_2+
\overline{Dr_2}\ell_2p_1+n)\overline{q}b_1}{p_1p_2}
\right)\nonumber\\&&\times
\sum_{b_2 \bmod q}
e\left(\frac{(\overline{Da_1}\ell_1p_2
-\overline{Da_2}\ell_2p_1+n)\overline{p_1}\overline{p_2} b_2}{q}
\right)
\mathfrak{J}\left(\frac{DXn}{p_1p_2q}\right)\nonumber\\
&=&
DX\mathop{\sum_{\overline{Dr_1}\ell_1p_2-
\overline{Dr_2}\ell_2p_1 \equiv n \bmod p_1p_2}}_{\overline{Da_1}\ell_1p_2
-\overline{Da_2}\ell_2p_1 +n\equiv 0 \bmod q}
\mathfrak{J}\left(\frac{DXn}{p_1p_2q}\right).
\ena
Thus
\bea\label{T-expression}
\mathbf{T}&=&DX
\sum_{\ell_1\in \mathcal{L}}\overline{\lambda_g(\ell_1)}
\sum_{\ell_2\in \mathcal{L}}\lambda_g(\ell_2)
\sum_{p_1\in\mathcal{P}}\chi(p_1)\overline{\xi_{D}}(p_1)
\sum_{p_2\in\mathcal{P}}\overline{\chi}(p_2)\xi_{D}(p_2)\nonumber\\
&&\sum_{0\neq |r_1|\ll R \atop (r_1,p_1)=1}
\sum_{0\neq |r_2|\ll R \atop (r_2,p_2)=1}\;
\sum_{\overline{r_1}\ell_1p_2-
\overline{r_2}\ell_2p_1 \equiv D n \bmod p_1p_2}\mathfrak{C}(n)
\mathfrak{J}\left(\frac{DXn}{p_1p_2q}\right),
\eea
where
\bea\label{C-expression}
\mathfrak{C}(n)=\sideset{}{^*}\sum_{a \bmod q \atop (\overline{a}\ell_1p_2+D n, q)=1}
\overline{\chi}(r_1+a)
\chi(r_2+\ell_2p_1\overline{(\overline{a}\ell_1p_2+D n)}).
\eea

\subsection{Analysis of the integral $\mathfrak{J}(x)$}

By \eqref{second derivative bound}, we have
\bea\label{trivial bound}
\mathfrak{J}(x)\ll t^{-1}.
\eea
We will use this estimate for smaller $x$. For larger $x$,
we wish to improve the above estimate by examining the triple integral more carefully.
Plugging \eqref{integral-J} into \eqref{J-definition}, we have
\bea\label{J-integral-0}
\mathfrak{J}(x)
&=&
\int_0^\infty\int_0^\infty
\widetilde{V}_{p_1,\ell_1}(v_1)\overline{\widetilde{V}_{p_2,\ell_2}(v_2)}
e\left(-\frac{t}{2\pi}\left(\log v_1-\log v_2\right)
-\frac{r_1Nv_1}{p_1q}+\frac{r_2Nv_2}{p_2q}\right)\nonumber\\
&&\qquad\qquad\qquad\qquad\qquad\qquad\times
\mathbf{H}\left(\frac{PK}{\sqrt{L}}
\left(\frac{\sqrt{v_1\ell_1}}{p_1}-
\frac{\sqrt{v_2\ell_2}}{p_2}\right), x\right)
\mathrm{d}v_1\mathrm{d}v_2,
\eea
where
\bna
\mathbf{H}(wK, x)=\int_0^\infty U(y)
e\left(2wK\sqrt{y}-xy\right)\mathrm{d}y.
\ena

We quote the following result of \cite{AHLQ}, Lemma 5.4.

\begin{lemma}\label{lemma: H integral}
Let $|r_1|,|r_2|\ll Pqt/N^{1-\varepsilon}$ and
$K>N^{\varepsilon}$. For real $w, x$ with $|w|\ll 1$, we have

(1) $\mathbf{H}(wK, x)=O_A(N^{-A})$ for $|x|\geq K^{1+\varepsilon}$ for any $\varepsilon>0$.

(2) For $|x|>N^\varepsilon$, we have $\mathbf{H}(wK, x)\ll N^{-A}$ unless
$2/3<wK/x<3/2$, say,
and for $1/2<wK/x<2$, if we let
$\rho=K^2w^2/x$ and
$W(\rho)=W(\rho, x)=e(-\rho)\mathbf{H}(\sqrt{\rho x}, x)$, then
\bna
\rho^jW^{(j)}(\rho)\ll_j |x|^{-1/2}.
\ena

(3) $\mathbf{H}(wK, 0)=W_0(2wK)$ for some Schwartz function $W_0$.

\end{lemma}

Following \cite{AHLQ}, we prove the following properties for $\mathfrak{J}(x)$.

\begin{lemma}\label{The integral J}
Assume $K< t^{1-\varepsilon}$.

(1)  We have $\mathfrak{J}(x)\ll_A N^{-A}$ if  $|x|\geq K^{1+\varepsilon}$ for any $\varepsilon>0$.

(2)  For $K^{2+\varepsilon}/t\ll |x|<K^{1+\varepsilon}$,
we have
\bea\label{large J-estimate}
\mathfrak{J}(x)\ll \frac{1}{t\sqrt{|x|}}.
\eea

(3)  Let $p_1=p_2=p\asymp P$.
Then
\bea\label{J0-estimate}
\mathfrak{J}(0)\ll \min
\left\{\frac{1}{t},
\frac{N^{\varepsilon}PLq}{K N |r_1\ell_2-r_2\ell_1|}\right\}.
\eea
Moreover, for $|r_1\ell_2-r_2\ell_2|\geq PLqt/(KN^{1-\varepsilon})$,
one has
\bea\label{N-small}
\mathfrak{J}(0)\ll_A \frac{N^{\varepsilon}}{K}
\left(\frac{PLq}{N|r_1\ell_2-r_2\ell_2|}\right)^{A}
\eea
for any $A\geq 0$.

\end{lemma}

\begin{proof}
(1)  The statement is obvious in view of Lemma \ref{lemma: H integral} (1).

(2)  By Lemma \ref{lemma: H integral} (2), we write the integral in \eqref{J-integral-0} as
\bna
\mathfrak{J}(x)
=\frac{1}{\sqrt{|x|}}
\int_0^\infty\int_0^\infty
\widetilde{V}_{p_1,\ell_1}(v_1)\overline{\widetilde{V}_{p_2,\ell_2}(v_2)}
W_1\left(wK/x\right)
e\left(h(v_1, v_2)\right)
\mathrm{d}v_1\mathrm{d}v_2
+O_A\left(N^{-A}\right),
\ena
where
$
w=PL^{-1/2}
\left(\sqrt{v_1\ell_1}/p_1-
\sqrt{v_2\ell_2}/p_2\right)
$,
$
W_1(y)=\sqrt{|x|}W(xy^2)F(y)
$
for $W$ defined as in Lemma \ref{lemma: H integral} (2)
satisfying $W_1^{(j)}(y)\ll_j 1$,
$F$ is a smooth function supported in $[1/2, 2]$
with $F\equiv 1$ on $[2/3, 3/2]$,
and
\bna
h(v_1, v_2)
&=&-\frac{t}{2\pi}\left(\log v_1-\log v_2\right)
-N\left(\frac{r_1v_1}{p_1q}-\frac{r_2v_2}{p_2q}\right)\\&&
-\frac{2K^2P^2\sqrt{v_1v_2\ell_1\ell_2}}{xLp_1p_2}
+\frac{K^2P^2}{xL}
\left(\frac{v_1\ell_1}{p_1^2}+\frac{v_2\ell_2}{p_2^2}\right).
\ena
By Fourier inversion, we write
\bna
W_1(y)=\int_{\mathbb{R}}\widehat{W_1}(v)e(vy)\mathrm{d}v,
\ena
where $\widehat{W_1}$ is the Fourier transform of $W_1$, satisfying $\widehat{W_1}(v)\ll(1+|v|)^{-A}$.
Then
\bna
\mathfrak{J}(x)=
\frac{1}{\sqrt{|x|}}
\int_{-N^\varepsilon}^{N^\varepsilon}\widehat{W_1}(v)
\int_0^\infty\int_0^\infty
\widetilde{V}_{p_1,\ell_1}(v_1)\overline{\widetilde{V}_{p_2,\ell_2}(v_2)}
e\left(h(v_1, v_2;v)\right)
\mathrm{d}v_1\mathrm{d}v_2\mathrm{d}v
+O_A\left(N^{-A}\right)
\ena
with
\bna
h(v_1, v_2; v)=h(v_1, v_2)+\frac{wKv}{x}
=h(v_1, v_2)+\frac{KPv}{x\sqrt{L}}
\left(\frac{\sqrt{v_1\ell_1}}{p_1}-\frac{\sqrt{v_2\ell_2}}{p_2}\right).
\ena
Note that for $K^{2+\varepsilon}/t\ll |x|<K^{1+\varepsilon}$,
\bna
\frac{\partial^2 h(v_1, v_2; v)}{\partial v_1^2}&=&
\frac{t}{2\pi v_1^2}
+\frac{K^2P^2\sqrt{v_2\ell_1\ell_2}}{2xLp_1p_2v_1^{3/2}}
-\frac{KPv\sqrt{\ell_1}}{4xp_1L^{1/2}v_1^{3/2}}\asymp t,
\\
\frac{\partial^2 h(v_1, v_2; v)}{\partial v_2^2}&=&
-\frac{t}{2\pi v_2^2}
+\frac{K^2P^2\sqrt{v_1\ell_1\ell_2}}{2xLp_1p_2v_2^{3/2}}
+\frac{KPv\sqrt{\ell_2}}{4xp_2L^{1/2}v_2^{3/2}}\asymp t,
\\
\frac{\partial^2 h(v_1, v_2; v)}{\partial v_1\partial v_2}&=&
-\frac{K^2P^2\sqrt{\ell_1\ell_2}}{2xLp_1p_2\sqrt{v_1v_2}}\asymp K^2.
\ena
Thus
\bna
|\det h''|=\left|\frac{\partial^2 h}{\partial v_1^2}
\frac{\partial^2 h}{\partial v_2^2}
-\frac{\partial^2 h}{\partial v_1\partial v_2}\right|\asymp  t^2
\ena
for $1\leq v_1,v_2\leq 2$ and $|v|\leq N^\varepsilon$.
By the two dimensional second derivative test in Lemma \ref{lem: 2nd derivative test, dim 2},
(2) follows.

(3)
Assume that
\bea\label{tem-assumption}
|r_1\ell_2-r_2\ell_1|\geq \frac{PqLt}{KN^{1-\varepsilon}}.
\eea
Otherwise we will use the estimate
$\mathfrak{J}(0)\ll t^{-1}$ which is already contained in \eqref{trivial bound}.
By \eqref{J-integral-0} and Lemma \ref{lemma: H integral} (3),
for $p_1=p_2=p$, we have
\bna
\mathfrak{J}(0)&=&
e^{it\log\ell_1/\ell_2}
\int_0^\infty\int_0^\infty \ell^{-1}_1\ell^{-1}_2
\widetilde{V}_{p,\ell_1}(v_1\ell^{-1}_1)
\overline{\widetilde{V}_{p,\ell_2}(v_2\ell^{-1}_2)}
e\left(-\frac{t}{2\pi}
\left(\log v_1-\log v_2\right)
\right)\\
&&
e\left(-\frac{r_1Nv_1}{p\ell_1q}+\frac{r_2Nv_2}{p\ell_2q}\right)
W_0\left(\frac{2PK}{p\sqrt{L}}
\left(\sqrt{v_1}-
\sqrt{v_2}\right), 0\right)
\mathrm{d}v_1\mathrm{d}v_2.
\ena
Changing variable
$\frac{P}{p\sqrt{L}}\left(\sqrt{v_1}-
\sqrt{v_2}\right)\rightarrow w$, one has
\bna
\mathfrak{J}(0)=
\int_{|w|<N^\varepsilon K^{-1}}W_0(2wK)
\int_0^{\infty}
V_0(w, v_2)e(f_0(w, v_2))\mathrm{d}v_2\mathrm{d}w
+O_A(N^{-A}),
\ena
where
\bna
V_0(w, v_2)=\frac{2p\sqrt{L}}{P\ell_1\ell_2}
\left(\frac{pw\sqrt{L}}{P}+\sqrt{v_2}\right)
\widetilde{V}_{p,\ell_1}\left(\frac{1}{\ell_1}
\left(\frac{pw\sqrt{L}}{P}+\sqrt{v_2}\right)^2
\right)
\overline{\widetilde{V}_{p,\ell_2}\left(\frac{v_2}{\ell_2}\right)}
\ena
satisfying
\bna
\mathbf{Var}\left(V_0(w, \cdot)\right)=\int_{\ell_2}^{2\ell_2}
\left|\frac{\partial V_0(w, v_2)}{\partial v_2}\right|
\mathrm{d}v_2\ll L^{-1},
\ena
and
\bna
f_0(w,v_2)=
-\frac{N(r_1\ell_1^{-1}-r_2\ell_2^{-1})}{pq}v_2
-\frac{t}{\pi}\log\left(1+\frac{pw\sqrt{L}}{P\sqrt{v_2}}\right)
-\frac{r_1N}{p\ell_1q}
\left(\frac{2pw\sqrt{Lv_2}}{P}+\frac{p^2w^2L}{P^2}\right).
\ena
For
$|r_1|,|r_2|\ll Pqt/N^{1-\varepsilon}$,
by the assumption \eqref{tem-assumption}, we have
\bna
\frac{\partial f_0(w, v_2)}{\partial v_2}
=-\frac{N(r_1\ell_1^{-1}-r_2\ell_2^{-1})}{pq}
+O\left(\frac{tN^\varepsilon}{LK}\right)
\ena
and
\bna
\frac{\partial f_0^2(w, v_2)}{\partial v_2^2}
\ll \frac{tN^\varepsilon}{KL^2}.
\ena
By integration by parts once, we get
\bna
\mathfrak{J}(0)\ll
\frac{N^{\varepsilon}PLq}{KNL|r_1\ell_2-r_2\ell_1|}.
\ena
Moreover, by Lemma \ref{lem: upper bound},
\bna
\mathfrak{J}(0)\ll_A \frac{N^{\varepsilon}}{K}
\left(\frac{PLq}{NL|r_1\ell_2-r_2\ell_1|}\right)^{A}
\ll_A \frac{N^{\varepsilon}}{K}
\left(\frac{PLq}{N|r_1\ell_2-r_2\ell_2|}\right)^{A}
\ena
for any $A\geq 0$.
This completes the proof of the lemma.

\end{proof}

\subsection{Estimation of $\mathscr{S}^{\sharp}(N)$}
Now we continue to estimate $\mathscr{S}^{\sharp}(N)$.
Recall \eqref{assumption 2} that $X={q}^2K^2P^2/(NL)$.
By Lemma \ref{The integral J}, the sum over $n$ in \eqref{T-expression} can be truncated at
$|n|\ll_D N^{1+\varepsilon}L/(qK)$ for any $\varepsilon>0$.
Thus $\mathbf{T}$ in \eqref{T-expression} is
\bna
\mathbf{T}&=&DX
\sum_{\ell_1\in \mathcal{L}}\overline{\lambda_g(\ell_1)}
\sum_{\ell_2\in \mathcal{L}}\lambda_g(\ell_2)
\sum_{p_1\in\mathcal{P}}\chi(p_1)\overline{\xi_{D}}(p_1)
\sum_{p_2\in\mathcal{P}}\overline{\chi}(p_2)\xi_{D}(p_2)\nonumber\\
&&\sum_{0\neq |r_1|\ll R \atop (r_1,p_1)=1}
\sum_{0\neq |r_2|\ll R \atop (r_2,p_2)=1}
\sum_{|n|\ll N^{1+\varepsilon}L/(qK) \atop
\overline{r_1}\ell_1p_2-
\overline{r_2}\ell_2p_1 \equiv D n \bmod p_1p_2}
\mathfrak{C}(n)
\mathfrak{J}\left(\frac{DXn}{p_1p_2q}\right)
+O(N^{-2020}).
\ena
When $\ell_1\neq \ell_2$, we apply the Cauchy-Schwarz inequality to
the $\ell_i$-sums to get rid of the Fourier coefficients
$\lambda_g(\ell_i)$ by using \eqref{Ranin-Selberg}.
Then
\bea\label{T-fen}
\mathbf{T}\ll_{g,\varepsilon}\mathbf{T}_0+\mathbf{T}_1+O_A(N^{-A}),
\eea
where
\bna
\mathbf{T}_0=X
\sum_{\ell\in \mathcal{L}}|\lambda_g(\ell)|^2
\mathop{\sum\sum}_{p_1\in\mathcal{P}\atop p_2\in\mathcal{P}}
\sum_{0\neq |r_1|\ll R \atop (r_1,p_1)=1}
\sum_{0\neq |r_2|\ll R \atop (r_2,p_2)=1}
\sum_{|n|\ll N^{1+\varepsilon}L/(qK) \atop
\overline{r_1}\ell_1p_2-
\overline{r_2}\ell_2p_1\equiv Dn \bmod p_1p_2}
|\mathfrak{C}(n)|\left|\mathfrak{J}\left(\frac{DXn}{p_1p_2q}\right)\right|
\ena
and
\bna
\mathbf{T}_1=XL^{1+\varepsilon}
\left(\mathop{\sum\sum}_{\ell_1,\ell_2\in \mathcal{L} \atop \ell_1\neq \ell_2}
\bigg(\mathop{\sum\sum}_{p_1\in\mathcal{P}\atop p_2\in\mathcal{P}}
\sum_{0\neq |r_1|\ll R \atop (r_1,p_1)=1}
\sum_{0\neq |r_2|\ll R \atop (r_2,p_2)=1}\sum_{|n|\ll N^{1+\varepsilon}L/(qK) \atop
\overline{r_1}\ell_1p_2-
\overline{r_2}\ell_2p_1\equiv Dn \bmod p_1p_2}
|\mathfrak{C}(n)|\left|\mathfrak{J}\left(\frac{DXn}{p_1p_2q}\right)\right|
\bigg)^2\right)^{1/2}.
\ena
We write $\mathbf{T}_0=\Delta_1+\Delta_2$
and $\mathbf{T}_1\ll\Sigma_1+\Sigma_2$,\
where $\Delta_1$ and $\Sigma_1$ are the contributions from the terms
with $n\equiv 0\bmod q$, and $\Delta_2$ and $\Sigma_2$ are
the contributions from the terms with $n\not\equiv 0\bmod q$,
with $\Delta_i$ and $\Sigma_j$ appropriately defined.

For $\mathfrak{C}(n)$, we quote the following results (see \cite{AHLS}).
\begin{lemma}\label{squareroot 2}
Let $q>3$ be a prime and we define
\begin{eqnarray*}
\mathfrak{C}=\sum_{z\in \mathbb{F}_q^{\times}
\atop (m+\gamma\overline{z},q)=1}\overline{\chi}(r_1+z)
\chi\bigg(r_2+\alpha(\overline{m+\gamma\overline{z}})\bigg),
\qquad (\alpha\gamma,q)=1,\quad m,r_1,r_2,\alpha,\gamma\in \mathbb{Z}.
\end{eqnarray*}
Suppose that $(r_1r_2,q)=1$. If $q|m$, we have
\begin{eqnarray*}
\mathfrak{C}=\chi(\alpha\overline{\gamma})
R_q(r_2-r_1\alpha\overline{\gamma})-\chi(r_2\overline{r_1}),
\end{eqnarray*}
where $R_q(a)=\sum_{z\in \mathbb{F}_q^{\times}}e(az/q)$
is the Ramanujan sum.
If $q\nmid  m$ and at least one of $r_1-\overline{m}\gamma$
and $r_2+\overline{m}\alpha$ is
nonzero in $\mathbb{F}_q$, then
\begin{equation*}
\mathfrak{C}\ll q^{1/2}.
\end{equation*}
Finally, if $m\neq0$ and $r_1-\overline{m}\gamma=r_2+\overline{m}\alpha=0$
in $\mathbb{F}_q$, then
\[
 \mathfrak{C}=
 \begin{cases}
 -\chi(mr_2\overline{\gamma}) \quad &\text{ if } \chi \text{ is not a quadratic character}, \\
 \chi(\overline{m}r_2\gamma)(q-1) \quad &\text{ if } \chi \text{ is a quadratic character}.
 \end{cases}
\]
\end{lemma}

According Lemma \ref{squareroot 2}, we further divide the analysis of
$\Delta_i$ and $\Sigma_j$ into cases.

\subsubsection{$n\equiv 0\bmod q$}

Let $\mathfrak{C}(n)$ be as \eqref{C-expression}.
By lemma \ref{squareroot 2}, we have
\bna
\mathfrak{C}(n)
=\left\{\begin{array}{ll}
O(q), &\mbox{if}\, r_2\ell_1p_2\equiv r_1\ell_2p_1 \bmod q,\\
O(1), &\mbox{otherwise}.
\end{array}\right.
\ena

According to $r_2\ell_1p_2\equiv r_1\ell_2p_1 \bmod q$ or not, we write
\bna
\Delta_1=\Delta_{10}+\Delta_{11},\qquad\qquad
\Sigma_1=\Sigma_{10}+\Sigma_{11},
\ena
where
\bna
\Delta_{10}=X
\sum_{\ell\in \mathcal{L}}|\lambda_g(\ell)|^2
\mathop{\sum\sum}_{p_1\in\mathcal{P}\atop p_2\in\mathcal{P}}
\mathop{\sum_{0\neq |r_1|\ll R \atop (r_1,p_1)=1}
\sum_{0\neq |r_2|\ll R \atop (r_2,p_2)=1}}_{r_2p_2\equiv r_1p_1\bmod q}
\mathop{\sum_{|n|\ll N^{1+\varepsilon}L/(qK) \atop
\overline{r_1}\ell_1p_2-
\overline{r_2}\ell_2p_1\equiv Dn \bmod p_1p_2}}_{n\equiv 0\bmod q}
q\left|\mathfrak{J}\left(\frac{DXn}{p_1p_2q}\right)\right|,
\ena
\bna
\Sigma_{10}=XL^{1+\varepsilon}
\left(\mathop{\sum\sum}_{\ell_1,\ell_2\in \mathcal{L}\atop \ell_1\neq \ell_2}
\bigg(\mathop{\sum\sum}_{p_1\in\mathcal{P}\atop p_2\in\mathcal{P}}
\mathop{\sum_{0\neq |r_1|\ll R \atop (r_1,p_1)=1}
\sum_{0\neq |r_2|\ll R \atop (r_2,p_2)=1}}_{r_2\ell_1p_2\equiv r_1\ell_2p_1\bmod q}
\mathop{\sum_{|n|\ll N^{1+\varepsilon}L/(qK) \atop
\overline{r_1}\ell_1p_2-
\overline{r_2}\ell_2p_1\equiv Dn \bmod p_1p_2}}_{n\equiv 0\bmod q}q
\left|\mathfrak{J}\left(\frac{DXn}{p_1p_2q}\right)
\right|\bigg)^2\right)^{1/2},
\ena
and $\Delta_{11}$, $\Sigma_{11}$ are the other pieces with the congruence
condition $r_2\ell_1p_2\not\equiv r_1\ell_2p_1 \bmod q$.

\begin{lemma}\label{Delta-1-0}
We have
\bna
\Delta_{10}+\Sigma_{10}\ll \frac{q^4P^4K^2}{N^{2-\varepsilon}}
+\frac{{q}^3tP^4L^2K^{1/2}}{N^{2-\varepsilon}}.
\ena
\end{lemma}

\begin{proof}
We distinguish two cases according as $p_1=p_2$ or not.

Case 1. $p_1=p_2=p$.

In this case, the congruence condition
$\overline{r_1}\ell_1p_2-\overline{r_2}\ell_2p_1 \equiv Dn\bmod p_1p_2$ implies $p|n$.
By the assumption \eqref{assumption 3} that
$P>N^{1+\varepsilon}L/(qK)$,
this in turn implies $n=0$ and  $r_2\ell_1\equiv r_1\ell_2 \bmod pq$.

First we consider the case $\ell_1\neq\ell_2$.
Recall \eqref{assumption 2} that we have assumed $K\geq t^{1/2}$ and $L\ll q^{1/4}$.
Thus if $r_1\ell_2\neq r_2\ell_1$, then
$
|r_1\ell_2-r_2\ell_1|\geq pq> PLqt/(KN^{1-\varepsilon}),
$
since $N\geq q^{3/4}t^{2/3}$ by \eqref{first estimate}.
By \eqref{N-small}, one has
\bna
\mathfrak{J}(0)\ll_A \frac{N^{\varepsilon}}{K}
\left(\frac{PLq}{N|r_1\ell_2-r_2\ell_2|}\right)^{A}\ll_A N^{-A}
\ena
for any $A>0$, which implies that the contribution from $r_1\ell_2\neq r_2\ell_1$
is negligibly small. If $r_1\ell_2=r_2\ell_1$, then $\ell_1|r_1$ and fixing
$r_1,\ell_1,\ell_2$ fixes $r_2$ uniquely. Therefore, by \eqref{Ranin-Selberg} and \eqref{J0-estimate},
the contribution from $p_1=p_2=p$ to $\Sigma_{10}$ is at most
\bea\label{p1 equal p2-1}
qXL^{1+\varepsilon}
\left(\mathop{\sum\sum}_{\ell_1,\ell_2\in \mathcal{L} \atop \ell_1\neq \ell_2}\bigg(\sum_{p\in\mathcal{P}}\;
\sum_{0\neq |r_1'|\ll R/L}
t^{-1}\bigg)^2\right)^{1/2}
\ll
\frac{qXL^{1+\varepsilon}LPR}{Lt}
\ll
\frac{q^4P^4K^2}{N^{2-\varepsilon}}
\eea
recalling \eqref{assumption 2} and \eqref{R definition} that $X={q}^2K^2P^2/(NL)$
and $R=Pqt/N^{1-\varepsilon}.$

Similarly, if $\ell_1=\ell_2$, then $r_1\equiv r_2 \bmod pq$.
By \eqref{N-small}, the contribution from $r_1\neq r_2$ is $O(N^{-2020})$ and
the contribution from $p_1=p_2=p$ to $\Delta_{10}$ is
\bea\label{p1 equal p2-0}
qX\sum_{\ell\in \mathcal{L}}|\lambda_g(\ell)|^2
\sum_{p\in\mathcal{P}}
\sum_{0\neq |r_1|\ll R \atop (r_1,p_1)=1}
\left|\mathfrak{J}(0)\right|+O_A(N^{-A})
\ll \frac{qXL^{1+\varepsilon}PR}{t}
\ll  \frac{q^4P^4K^2}{N^{2-\varepsilon}}.
\eea

Case 2. $p_1\neq p_2$.

In this case, we have $(n, p_1p_2)=1$ and $r_1\equiv \overline{Dn}\ell_1p_2\bmod p_1$
and $r_2\equiv -\overline{D n}\ell_2p_1\bmod p_2$.
Note that $n\equiv0\bmod q$ implies that $|n|\geq q$ and
$Xn/(p_1p_2q)\gg q^2K^2/(NL)\gg K^{2+\varepsilon}/t$. By
\eqref{large J-estimate}, we have
\bna
\mathfrak{J}\left(\frac{Xn}{p_1p_2q}\right)\ll
\frac{P{q}^{1/2}}{tX^{1/2}|n|^{1/2}}.
\ena
Thus the contribution from $p_1\neq p_2$ to $\Sigma_{10}$ is bounded by
\bea\label{p1 not equal p2-1}
&&qXL^{1+\varepsilon}\frac{P{q}^{1/2}}{tX^{1/2}}
\left(\mathop{\sum\sum}_{\ell_1,\ell_2\in \mathcal{L} \atop \ell_1\neq \ell_2}
\bigg(\mathop{\sum\sum}_{p_1\in\mathcal{P}\atop p_2\in\mathcal{P}}\;
\sum_{0\neq|n|\ll N^{1+\varepsilon}L/(qK)\atop
n\equiv 0\bmod q}|n|^{-1/2}
\sum_{0\neq |r_1|\ll R \atop r_1\equiv \overline{Dn}\ell_1p_2\bmod p_1}
\sum_{0\neq |r_2|\ll R \atop r_2\equiv -\overline{Dn}\ell_2p_1\bmod p_2}1
\bigg)^2\right)^{1/2}\nonumber\\
&\ll&
\frac{{q}^{3/2}X^{1/2}PL^{1+\varepsilon}}{t}P^2L\left(\frac{R}{P}\right)^2
\left(\frac{NL}{{q}^2K}\right)^{1/2}
\frac{1}{{q}^{1/2}}\nonumber\\
&\ll&\frac{{q}^3tP^4L^{2+\varepsilon}K^{1/2}}{N^2}
\eea
recalling \eqref{assumption 2} and \eqref{R definition}.

Similarly, the contribution from $p_1\neq p_2$ to $\Delta_{10}$ is at most
\bea\label{p1 not equal p2-0}
&&qX\frac{P{q}^{1/2}}{tX^{1/2}}\sum_{\ell\in \mathcal{L}}
|\lambda_g(\ell)|^2
\mathop{\sum\sum}_{p_1\in\mathcal{P}\atop p_2\in\mathcal{P}}\;
\sum_{0\neq|n|\ll N^{1+\varepsilon}L/(qK)\atop
n\equiv 0\bmod q}|n|^{-1/2}
\sum_{0\neq |r_1|\ll R \atop r_1\equiv \overline{Dn}\ell_1p_2\bmod p_1}
\sum_{0\neq |r_2|\ll R \atop r_2\equiv -\overline{Dn}\ell_2p_1\bmod p_2}1\nonumber\\
&\ll&
\frac{{q}^{3/2}X^{1/2}P}{t}P^2L^{1+\varepsilon}\left(\frac{R}{P}\right)^2
\left(\frac{NL}{{q}^2K}\right)^{1/2}
\frac{1}{{q}^{1/2}}\nonumber\\
&\ll&
\frac{{q}^3tP^4L^{1+\varepsilon}K^{1/2}}{N^2}.
\eea
By \eqref{p1 equal p2-1}-\eqref{p1 not equal p2-0}, we conclude that
\bna
\Delta_{10}+\Sigma_{10}\ll \frac{q^4P^4K^2}{N^{2-\varepsilon}}+
\frac{{q}^3tP^4L^2K^{1/2}}{N^{2-\varepsilon}}.
\ena
The proves the lemma.

\end{proof}

\begin{lemma}\label{Delta-1-1}
We have
\bna
\Delta_{11}+\Sigma_{11}\ll \frac{q^2tP^4L^2K^{1/2}}{N^{2-\varepsilon}}+
\frac{q^3P^4K^2}{N^{2-\varepsilon}}
+\frac{{q}^4tP^4KL}{N^{3-\varepsilon}}.
\ena
\end{lemma}

\begin{proof} As in Lemma \ref{Delta-1-0}, we distinguish two cases according as $p_1=p_2$ or not.

Case 1. $p_1=p_2=p$.

In this case, the congruence $\overline{r_1}\ell_1p_2-
\overline{r_2}\ell_2p_1\equiv Dn \bmod p_1p_2$
implies $n=0$, since $p|n$ and
$|n|\ll N^{1+\varepsilon}L/qK<P$ by \eqref{assumption 3}. Thus by \eqref{J0-estimate},
the contribution from
$p_1=p_2=p$ to $\Sigma_{11}$ is at most
\bea\label{11-1}
&&XL^{1+\varepsilon}
\left(\mathop{\sum\sum}_{\ell_1,\ell_2\in \mathcal{L} \atop \ell_1\neq \ell_2}
\bigg(\sum_{p\in\mathcal{P}}
\sum_{0\neq |r_1|\ll R \atop \ell_1|r_1}\frac{1}{t}
+\sum_{p\in\mathcal{P}}\;
\sum_{0\neq |r_1|\ll R}
\mathop{\sum_{0\neq |r_2|\ll R \atop r_1\ell_2\equiv r_2\ell_1 \bmod p}}_{r_1\ell_2\neq r_2\ell_1}
\frac{PLq}{KN|r_1\ell_2-r_2\ell_1|}
\bigg)^2\right)^{1/2}\nonumber\\
&\ll&\frac{XL^{1+\varepsilon}PR}{t}+\frac{qXL^{2+\varepsilon}P}{KN}
\left(\mathop{\sum\sum}_{\ell_1,\ell_2\in \mathcal{L} \atop \ell_1\neq \ell_2}
\bigg(\sum_{p\in\mathcal{P}}\;\sum_{0\neq d\ll RL/P}
\frac{1}{|d|p}
\sum_{0\neq |r_2|\ll R \atop r_2\ell_1+dp \equiv 0\bmod \ell_2}
1
\bigg)^2\right)^{1/2}\nonumber\\
&\ll&\frac{XL^{1+\varepsilon}PR}{t}+\frac{qXL^{2+\varepsilon}PR}{KN}\nonumber\\
&\ll&\frac{q^3P^4K^2}{N^{2-\varepsilon}}+
\frac{{q}^4tP^4KL}{N^{3-\varepsilon}}
\eea
recalling \eqref{assumption 2} and \eqref{R definition}.

Similarly,  by \eqref{J0-estimate},  the contribution from $p_1=p_2=p$ to $\Delta_{11}$
is at most
\bea\label{11-2}
&&X\sum_{\ell\in \mathcal{L}}|\lambda_g(\ell)|^2\sum_{p\in\mathcal{P}}
\left(\sum_{0\neq |r_1|\ll R}\frac{1}{t}
+\sum_{0\neq |r_1|\ll R}
\mathop{\sum_{0\neq |r_2|\ll R \atop r_1\equiv r_2 \bmod p}}_{r_1\neq r_2}
\frac{Pq}{KN|r_1-r_2|}\right)\nonumber\\
&\ll&\frac{XL^{1+\varepsilon}PR}{t}+
\frac{qXL^{1+\varepsilon}PR}{KN}
\sum_{p\in\mathcal{P}}\sum_{0\neq d\ll R/p}\frac{1}{|d|p}\nonumber\\
&\ll&\frac{q^3P^4K^2}{N^{2-\varepsilon}}+
\frac{{q}^4tP^4K}{N^{3-\varepsilon}}.
\eea

Case 2. $p_1\neq p_2$.

In this case, the congruence condition implies that
$r_1\equiv \overline{Dn}\ell_2p_2 \bmod p_1$ and
$r_2\equiv -\overline{Dn}\ell_2p_1 \bmod p_2$. Thus by
\eqref{large J-estimate}, the contribution from
$p_1\neq p_2$ to $\Sigma_{11}$ is at most
\bea\label{11-3}
&&XL^{1+\varepsilon}
\left(\mathop{\sum\sum}_{\ell_1,\ell_2\in \mathcal{L} \atop \ell_1\neq \ell_2}
\bigg(\mathop{\sum\sum}_{p_1, p_2\in \mathcal{P} \atop p_1\neq p_2}
\sum_{0\neq|n|\ll N^{1+\varepsilon}L/qK
\atop n\equiv 0 \bmod q}
\sum_{0\neq |r_1|\ll R \atop r_1\equiv\overline{Dn}\ell_1p_2 \bmod p_1}
\sum_{0\neq |r_2|\ll R \atop r_2\equiv-\overline{Dn}\ell_2p_1 \bmod p_2}
\left|\mathfrak{J}\left(\frac{DXn}{p_1p_2q}\right)
\right|\bigg)^2\right)^{1/2}\nonumber\\
&\ll&XL^{1+\varepsilon}
\left(\mathop{\sum\sum}_{\ell_1,\ell_2\in \mathcal{L} \atop \ell_1\neq \ell_2}
\bigg(\mathop{\sum\sum}_{p_1, p_2\in \mathcal{P} \atop p_1\neq p_2}
\sum_{0\neq|n|\ll N^{1+\varepsilon}L/(qK)
\atop n\equiv 0 \bmod q}
\left(\frac{R}{p}\right)^2
\frac{Pq^{1/2}}{tX^{1/2}|n|^{1/2}}
\bigg)^2\right)^{1/2}\nonumber\\
&\ll&
\frac{{q}^{1/2}X^{1/2}L^{2+\varepsilon}R^2P}{t}
\sum_{0\neq|n'|\ll N^{1+\varepsilon}L/(q^2K)}
(|n'|q)^{-1/2}\nonumber\\
&\ll&
\frac{X^{1/2}L^{2+\varepsilon}R^2P}{t}
\left(\frac{NL}{{q}^2K}\right)^{1/2}\nonumber\\
&\ll&
\frac{q^2tP^4L^2K^{1/2}}{N^{2-\varepsilon}}
\eea
recalling \eqref{assumption 2} and \eqref{R definition}.

Similarly, by
\eqref{large J-estimate}, the contribution from
$p_1\neq p_2$ to $\Delta_{11}$ is at most
\bea\label{11-4}
&&X\sum_{\ell\in \mathcal{L}}|\lambda_g(\ell)|^2\mathop{\sum\sum}_{p_1, p_2\in \mathcal{P} \atop p_1\neq p_2}\;
\sum_{0\neq|n'|\ll N^{1+\varepsilon}L/(q^2K)}
\sum_{0\neq |r_1|\ll R \atop r_1\equiv\overline{Dn}\ell p_2 \bmod p_1}
\sum_{0\neq |r_2|\ll R \atop r_2\equiv-\overline{Dn}\ell p_1 \bmod p_2}
\frac{Pq^{1/2}}{t(|n'|qX)^{1/2}}\nonumber\\
&\ll&\frac{X^{1/2}P}{t}LP^2
\left(\frac{N^{1+\varepsilon}L}{{q}^2K}\right)^{1/2}
\left(\frac{R}{P}\right)^2\nonumber\\
&\ll&
\frac{q^2tP^4LK^{1/2}}{N^{2-\varepsilon}}.
\eea
By \eqref{11-1}-\eqref{11-4}, we conclude that
\bna
\Delta_{11}+\Sigma_{11}\ll \frac{q^2tP^4L^2K^{1/2}}{N^{2-\varepsilon}}+ \frac{q^4P^4K^2}{N^{2-\varepsilon}}
+\frac{{q}^5tP^4KL}{N^{3-\varepsilon}}.
\ena
The proves the lemma.

\end{proof}

\subsubsection{$n\not\equiv 0\bmod q$}

Let $\mathfrak{C}(n)$ be as \eqref{C-expression}.
By Lemma \ref{squareroot 2}, we have
\bna
\mathfrak{C}(n)
=\left\{\begin{array}{ll}
O(q), &\mbox{if}\; r_1-\overline{Dn}\ell_1p_2\equiv r_2+\overline{Dn}\ell_2p_1\equiv 0
 \bmod q,\\
O(q^{1/2}), &\mbox{otherwise}.
\end{array}\right.
\ena
According to $r_1-\overline{Dn}\ell_1p_2\equiv r_2+\overline{Dn}\ell_2p_1\equiv 0
 \bmod q$ or not, we write
\bna
\Delta_2=\Delta_{20}+\Delta_{21},\qquad\qquad
\Sigma_2=\Sigma_{20}+\Sigma_{21},
\ena
where
\bna
\Delta_{20}=X
\sum_{\ell\in \mathcal{L}}|\lambda_g(\ell)|^2
\mathop{\sum\sum}_{p_1\in\mathcal{P}\atop p_2\in\mathcal{P}}
\sum_{0\neq |r_1|\ll R \atop (r_1,p_1)=1}
\sum_{0\neq |r_2|\ll R \atop (r_2,p_2)=1}
\mathop{\sum_{0\neq|n|\ll N^{1+\varepsilon}L/(qK) \atop
n\not\equiv 0\bmod q }}_{
\overline{r_1}\ell  p_2-
\overline{r_2}\ell  p_1 \equiv Dn \bmod p_1p_2
\atop r_1-\overline{Dn}\ell  p_2\equiv r_2+\overline{Dn}\ell  p_1\equiv 0
 \bmod q}
q\left|\mathfrak{J}\left(\frac{DXn}{p_1p_2q}\right)
\right|,
\ena
\bna
\Sigma_{20}=XL^{1+\varepsilon}
\left(\mathop{\sum\sum}_{\ell_1,\ell_2\in \mathcal{L}\atop \ell_1\neq \ell_2}
\bigg(\mathop{\sum\sum}_{p_1\in\mathcal{P}\atop p_2\in\mathcal{P}}
\sum_{0\neq |r_1|\ll R \atop (r_1,p_1)=1}
\sum_{0\neq |r_2|\ll R \atop (r_2,p_2)=1}
\mathop{\sum_{0\neq|n|\ll N^{1+\varepsilon}L/(qK) \atop
n\not\equiv 0\bmod q }}_{
\overline{r_1}\ell_1  p_2-
\overline{r_2}\ell_2  p_1 \equiv Dn \bmod p_1p_2
\atop r_1-\overline{Dn}\ell_1  p_2\equiv r_2+\overline{Dn}\ell_2  p_1\equiv 0
 \bmod q}q\left|\mathfrak{J}\left(\frac{DXn}{p_1p_2q}\right)
\right|\bigg)^2\right)^{1/2},
\ena
and $\Delta_{21}$, $\Sigma_{21}$ are the other pieces.

\begin{lemma}\label{Delta-2-0}
We have
\bna
\Delta_{20}+\Sigma_{20}\ll \frac{q^4P^4K^2}{N^{2-\varepsilon}}+
\frac{q^2P^4L^2K^{1/2}}{N^{1-\varepsilon}}\left(1+\frac{t}{N}\right).
\ena
\end{lemma}

\begin{proof}
First we note that if $p_1=p_2=p$, then $p|n$,
which is impossible by \eqref{assumption 3}. Thus
$p_1\neq p_2$ and $(n,p_1p_2)=1$.
Applying \eqref{trivial bound} and \eqref{large J-estimate}, we get
\bna
\Sigma_{20}\ll \Sigma_{20}^*+\Sigma_{20}^{**},
\ena
where
\bna
\Sigma_{20}^*&=&qXL^{1+\varepsilon}
\left(\mathop{\sum\sum}_{\ell_1,\ell_2\in \mathcal{L}\atop \ell_1\neq \ell_2}
\bigg(\mathop{\sum\sum}_{p_1,p_2\in\mathcal{P}\atop p_1\neq p_2}
\sum_{0\neq|n|\ll N^{1+\varepsilon}L/(qt)\atop (n,p_1p_2q)=1}
\sum_{0\neq |r_1|\ll R \atop Dnr_1\equiv \ell_1p_2\bmod p_1q}
\sum_{0\neq |r_2|\ll R \atop Dnr_2\equiv -\ell_2p_1\bmod p_2q}
\frac{1}{t}\bigg)^2\right)^{1/2}
\ena
and
\bna
\Sigma_{20}^{**}&=&
qXL^{1+\varepsilon}\\
&\times&\left(\mathop{\sum\sum}_{\ell_1,\ell_2\in \mathcal{L}\atop \ell_1\neq \ell_2}
\bigg(\mathop{\sum\sum}_{p_1,p_2\in\mathcal{P}\atop p_1\neq p_2}
\sum_{\frac{N^{1+\varepsilon}L}{qt}\ll|n|\ll\frac{N^{1+\varepsilon}L}{qK}
\atop (n,p_1p_2q)=1}
\sum_{0\neq |r_1|\ll R \atop Dnr_1\equiv \ell_1p_2\bmod p_1q}
\sum_{0\neq |r_2|\ll R \atop Dnr_2\equiv -\ell_2p_1\bmod p_2q}
\frac{Pq^{1/2}}{tX^{1/2}|n|^{1/2}}
\bigg)^2\right)^{1/2}.
\ena

We first estimate $\Sigma_{20}^*$. Recall \eqref{R definition} that
$R=Pqt/N^{1-\varepsilon}.$
We have
$|DnR|\ll\frac{N^{1+\varepsilon}L}{qt}
\cdot \frac{Pqt}{N^{1-\varepsilon}}\ll N^{\varepsilon}PL<Pq$.
Thus the congruence conditions give equalities
$Dnr_1=\ell_1p_2$ and $Dnr_2=-\ell_2p_1$.
Moreover, $Dn=\ell_1p_2/r_1=-\ell_2p_1/r_2$
implies $\ell_1p_2r_2=-\ell_2p_1r_1$. Therefore fixing $\ell_1, p_2, r_2$ fixes $\ell_2, p_1, r_1$ up to factors of $\log q$.
Similarly, $n'=\ell_1p_2'/r_1'=-\ell_2p_1'/r_2'$,
so that $\ell_1p_2'r_2'=-\ell_2p_1'r_1'$ which implies that
$\ell_2|r_2'$, since $\ell_1\neq \ell_2$. Consequently,
\bea\label{Sigma20-1}
\Sigma_{20}^*&\ll&\frac{qXL^{1+\varepsilon}}{t}
\left(\sum_{\ell_1\in  \mathcal{L}}
\sum_{p_2\in\mathcal{P}}\sum_{0\neq |r_2|\ll R}\;
\mathop{\sum_{\ell_2\neq \ell_1}
\sum_{p_1\in\mathcal{P}}\sum_{0\neq |r_1|\ll R}}_{ \ell_1p_2r_2=-\ell_2p_1r_1}\;
\sum_{n=\ell_1p_2/(Dr_1)}\right.\nonumber\\
&&\qquad\qquad
\left.
\sum_{p_2'\in\mathcal{P}}\sum_{0\neq |r_2'|\ll R \atop \ell_2|r_2'}
\mathop{\sum_{p_1'}\sum_{r_1'}}_{p_1'r_1'=-\ell_1p_2'r_2'/\ell_2}\;
\sum_{n'=\ell_1p_2'/(Dr_1')}1
\right)^{1/2}\nonumber\\
&\ll&
\frac{qXL^{1+\varepsilon}}{t}
\left(LPR\cdot P\frac{R}{L}\right)^{1/2}\nonumber\\
&\ll&\frac{q^4P^4K^2}{N^{2-\varepsilon}}
\eea
recalling \eqref{assumption 2} and \eqref{R definition}.

For $\Sigma_{20}^{**}$, we make a dyadic subdivision to the sum over $n$ to write it as
\bna
\sum_{\frac{N^{1+\varepsilon}L}{qt}\ll N_1\ll\frac{N^{1+\varepsilon}L}{qK}\atop
\mathrm{dyadic}}
\frac{q^{3/2}X^{1/2}PL^{1+\varepsilon}}{tN_1^{1/2}}
\left(\mathop{\sum\sum}_{\ell_1,\ell_2\in \mathcal{L}\atop \ell_1\neq \ell_2}
\bigg(\mathop{\sum\sum}_{p_1,p_2\in\mathcal{P}\atop p_1\neq p_2}
\sum_{n\asymp N_1\atop (n,p_1p_2q)=1}
\sum_{0\neq |r_1|\ll R \atop Dnr_1\equiv \ell_1p_2\bmod p_1q}
\sum_{0\neq |r_2|\ll R \atop Dnr_2\equiv -\ell_2p_1\bmod p_2q}
1\bigg)^2\right)^{1/2}.
\ena
If $RN_1<Pq$, then
the congruence conditions give equalities
$Dnr_1=\ell_1p_2$ and $Dnr_2=-\ell_2p_1$.
Similarly as the proof of \eqref{Sigma20-1},
the contribution from $N_1<Pq/R$ to  $\Sigma_{20}^{**}$ is bounded by
\bea\label{Sigma20-2}
&&\sup_{\frac{N^{1+\varepsilon}L}{qt}\ll N_1\ll \frac{N^{1-\varepsilon}}{t}}
\frac{q^{3/2}X^{1/2}PL^{1+\varepsilon}}{tN_1^{1/2}}
\left(LPR\cdot P\cdot \frac{R}{L}\right)^{1/2}\nonumber\\
&&\ll
\frac{q^{3/2}X^{1/2}P^2RL^{1+\varepsilon}}{t}\left(\frac{qt}{NL}\right)^{1/2}\nonumber\\
&&\ll \frac{q^4t^{1/2}P^4K}{N^{2-\varepsilon}}.
\eea
If $RN_1\geq Pq$, we rewrite it as
\bea\label{Sigma20-3}
&&\sum_{\frac{N^{1+\varepsilon}L}{qt}\ll N_1\ll
\frac{N^{1+\varepsilon}L}{qK}\atop
\mathrm{dyadic}}
\frac{q^{3/2}X^{1/2}PL^{1+\varepsilon}}{tN_1^{1/2}}
\left(\mathop{\sum\sum}_{\ell_1,\ell_2\in \mathcal{L}\atop \ell_1\neq \ell_2}
\bigg(\mathop{\sum\sum}_{p_1,p_2\in\mathcal{P}\atop p_1\neq p_2}
\sum_{0\neq |\widetilde{r}|\ll RN_1 \atop \widetilde{r}\equiv \ell_1p_2\bmod p_1q}
\sum_{n\asymp N_1\atop n|\widetilde{r}}
\sum_{0\neq |r_2|\ll R \atop r_2\equiv -\overline{Dn}\ell_2p_1\bmod p_2q}
1\bigg)^2\right)^{1/2}\nonumber\\
&&\ll
\sum_{\frac{N^{1+\varepsilon}L}{qt}\ll N_1\ll\frac{N^{1+\varepsilon}L}{qK}\atop
\mathrm{dyadic}}
\frac{q^{3/2}X^{1/2}PL^{1+\varepsilon}}{tN_1^{1/2}}LP^2\frac{RN_1}{Pq}
\left(1+\frac{R}{Pq}\right)\nonumber\\
&&\ll \frac{q^2P^4L^2K^{1/2}}{N^{1-\varepsilon}}\left(1+\frac{t}{N}\right).
\eea
By \eqref{Sigma20-1}-\eqref{Sigma20-3}, we conclude that
\bna
\Sigma_{20}\ll \frac{q^4P^4K^2}{N^{2-\varepsilon}}+
\frac{q^4t^{1/2}P^4K}{N^{2-\varepsilon}}+
\frac{q^2P^4L^2K^{1/2}}{N^{1-\varepsilon}}\left(1+\frac{t}{N}\right).
\ena
Note that the first term dominates the second term by the assumption in \eqref{assumption 2}.
Thus
\bna
\Sigma_{20}\ll \frac{q^4P^4K^2}{N^{2-\varepsilon}}+
\frac{q^2P^4L^2K^{1/2}}{N^{1-\varepsilon}}\left(1+\frac{t}{N}\right).
\ena
Similarly,
\bna
\Delta_{20}
\ll\frac{q^4P^4K^2}{N^{2-\varepsilon}}+
\frac{q^2P^4LK^{1/2}}{N^{1-\varepsilon}}\left(1+\frac{t}{N}\right).
\ena
Therefore,
\bna
\Delta_{20}+\Sigma_{20}\ll \frac{q^4P^4K^2}{N^{2-\varepsilon}}+
\frac{q^2P^4L^2K^{1/2}}{N^{1-\varepsilon}}\left(1+\frac{t}{N}\right).
\ena
The lemma follows.
\end{proof}

\begin{lemma}\label{Delta-2-1}
We have
\bna
\Delta_{21}+\Sigma_{21}\ll\frac{{q}^{7/2}P^4K^2L^2}{N^{2-\varepsilon}}
\left(1+\frac{t}{K^{3/2}}\right).
\ena
\end{lemma}

\begin{proof}
As in the proof of Lemma \ref{Delta-2-0}, we have $p_1\neq p_2$ and $(n,p_1p_2q)=1$.
Moreover, for fixed $n$, $p_i$ and $\ell_i, i=1,2$,
$r_1\equiv \overline{Dn}\ell_1p_2\bmod p_1$ and
$r_2\equiv -\overline{Dn}\ell_2p_1\bmod p_2$.
Applying \eqref{trivial bound} and \eqref{large J-estimate}, we get
\bna
\Sigma_{21}&\ll& {q}^{1/2}XL^{1+\varepsilon}
\left(\mathop{\sum_{\ell_1\in L}\sum_{\ell_2\in L}}_{\ell_1\neq \ell_2}
\left(\mathop{\sum\sum}_{p_1,p_2\in\mathcal{P}\atop p_1\neq p_2}
\sum_{0\neq|n|\ll\frac{N^{1+\varepsilon}L}{qt}}
\sum_{0\neq |r_1|\ll R \atop r_1\equiv \overline{Dn}\ell_1p_2\bmod p_1}
\sum_{0\neq |r_2|\ll R \atop r_2\equiv -\overline{Dn}\ell_2p_1\bmod p_2}
\frac{1}{t}\right.\right.\\
&+&
\left.\left.
\mathop{\sum\sum}_{p_1,p_2\in\mathcal{P}\atop p_1\neq p_2}
\sum_{\frac{N^{1+\varepsilon}L}{qt}\ll|n|\ll\frac{N^{1+\varepsilon}L}{qK}}
\sum_{0\neq |r_1|\ll R \atop r_1\equiv \overline{Dn}\ell_1p_2\bmod p_1}
\sum_{0\neq |r_2|\ll R \atop r_2\equiv -\overline{Dn}\ell_2p_1\bmod p_2}
\frac{Pq^{1/2}}{tX^{1/2}|n|^{1/2}}
\right)^2\right)^{1/2}\\
&\ll&N^{\varepsilon}
{q}^{1/2}XL^2P^2
\left(1+\frac{R}{P}\right)^2
\left(\frac{NL}{qt^2}+\frac{Pq^{1/2}}{tX^{1/2}}
\left(\frac{NL}{qK}\right)^{1/2}\right)\\
&\ll&
N^{\varepsilon}\frac{{q}^{7/2}P^4K^2L^2}{N^2}
\left(1+\frac{t}{K^{3/2}}\right).
\ena
Similarly,
\bna
\Delta_{21}\ll
\frac{{q}^{7/2}P^4K^2L}{N^{2-\varepsilon}}
\left(1+\frac{t}{K^{3/2}}\right).
\ena
Therefore,
\bna
\Delta_{21}+\Sigma_{21}\ll
\frac{{q}^{7/2}P^4K^2L^2}{N^{2-\varepsilon}}
\left(1+\frac{t}{K^{3/2}}\right).
\ena

\end{proof}

Putting the bounds of Lemmas \ref{Delta-1-0}-\ref{Delta-2-1} into \eqref{T-fen},
we have
\bna
\mathbf{T}&\ll_{g,\varepsilon}&\frac{q^4P^4K^2}{N^{2-\varepsilon}}+
\frac{q^2tP^4L^2K^{1/2}}{N^{2-\varepsilon}}
+\frac{{q}^4tP^4KL}{N^{3-\varepsilon}}\\&&+
\frac{q^2P^4L^2K^{1/2}}{N^{1-\varepsilon}}\left(1+\frac{t}{N}\right)+
\frac{{q}^{7/2}P^4K^2L^2}{N^{2-\varepsilon}}
\left(1+\frac{t}{K^{3/2}}\right).
\ena
By taking $K=t^{2/3}$ we get
\bna
\mathbf{T}&\ll_{g,\varepsilon}&\frac{q^4t^{4/3}P^4}{N^{2-\varepsilon}}+
\frac{q^2t^{4/3}P^4L^2}{N^{2-\varepsilon}}
+\frac{{q}^4t^{5/3}P^4L}{N^{3-\varepsilon}}\\&&+
\frac{q^2t^{1/3}P^4L^2}{N^{1-\varepsilon}}\left(1+\frac{t}{N}\right)+
\frac{{q}^{7/2}t^{4/3}P^4L^2}{N^{2-\varepsilon}}.
\ena
Note that the first term dominates the second term, the third term and the fourth term since
$q^{3/4}t^{2/3}< N\leq (qt)^{1+\varepsilon}$ and we will choose $L\leq q^{1/4}$.
Thus
\bna
\mathbf{T}\ll_{g,\varepsilon}\frac{q^4t^{4/3}P^4}{N^{2-\varepsilon}}
+
\frac{q^{7/2}t^{4/3}P^4L^2}{N^{2-\varepsilon}}.
\ena
This estimate when plugged into \eqref{S-sharp-inequality} yields that
\bna
\mathscr{S}^{\sharp}(N)&\ll_{g,\varepsilon}&
\frac{N^{3/2}}{{q}^{3/2}P^2t^{1/3}L^{1/2}}\left(
\frac{q^2t^{2/3}P^2}{N}
+\frac{q^{7/4}t^{2/3}P^2L}{N}\right)\\
&\ll&N^{1/2}\left(
\frac{q^{1/2}t^{1/3}}{L^{1/2}}
+q^{1/4}t^{1/3}L^{1/2}\right).
\ena
Taking $L={q}^{1/4}$ to balance the two terms, we obtain
\bea\label{S-sharp-1}
\mathscr{S}^{\sharp}(N)\ll_{g,\varepsilon}
N^{1/2}q^{3/8}t^{1/3}.
\eea

\subsection{Estimates for $\mathscr{S}_{c=q}(N)$}
The term $\mathscr{S}_{c=q}(N)$ can be estimated very similarly as $\mathscr{S}^{\sharp}(N)$.
For $K=t^{2/3}$ and $L={q}^{1/4}$, one has
\bna
\mathscr{S}_{c=q}(N)\ll_{g,\varepsilon}
\frac{N^{1/2}}{P}\left(
q^{1/8}t^{1/2}
+q^{3/8}t^{1/3}\right).
\ena
Taking $P=q^{1/4+\varepsilon}t^{1/3+\varepsilon}$. It is easily seen that
$P$ satisfies the assumption in \eqref{assumption 3}. For this choice of $P$, we
have
\bea\label{S-c-q-1}
\mathscr{S}_{c=q}(N)\ll_{g,\varepsilon}
N^{1/2}\left(
q^{-1/8}t^{1/6}
+q^{1/8}\right).
\eea

\subsection{Estimates for $\mathscr{S}_{c=p}(N)$}

Let $c|pq$.
Notice that by Fourier inversion, we can write
\bea\label{expression}
U\left(\frac{p^2q^2n}{DXc^2}\right)\mathcal{J}(n,r,\ell;c,p)
=\int_{\mathbb{R}}\widehat{\mathcal{J}}(x,r,\ell;c,p)e(nx)\mathrm{d}x,
\eea
where
\bna
\widehat{\mathcal{J}}(x,r,\ell;c,p):=\int_{\mathbb{R}}
U\left(\frac{p^2q^2u}{DXc^2}\right)\mathcal{J}(u,r,\ell;c,p)
e(-xu)\mathrm{d}u.
\ena
with $\mathcal{J}(u,r,\ell;c,p)$ defined as in \eqref{integral-J}.
First we claim that the range of integration in $x$
in \eqref{expression}
can be restricted to $|x|\leq N^{\varepsilon}p^2q^2K/(c^2X)$.
To see this, plugging \eqref{integral-J} in, one has
\bea\label{J-hat}
\widehat{\mathcal{J}}(x,r,\ell;c,p)&=&
\frac{DXc^2}{p^2q^2}\int_0^\infty\widetilde{V}_{p,\ell}(y)
e\left(-\frac{t}{2\pi}\log y
-\frac{rNy}{[c,q]}\right)\nonumber\\&&
\times\int_{\mathbb{R}}
U\left(u\right)
e\left(\frac{2\sqrt{NX\ell yu}}{pq}-\frac{DXc^2u}{p^2q^2}x\right)\mathrm{d}u
\mathrm{d}y
\eea
where by applying repeated integration
by parts,
\bna
&&\int_{\mathbb{R}}
U\left(u\right)
e\left(\frac{2\sqrt{NX\ell yu}}{pq}-\frac{DXc^2u}{p^2q^2}x\right)\mathrm{d}u\\
&\ll_{D,j}&\left(\frac{p^2q^2}{|x|c^2X}\left(1+\frac{\sqrt{NXL}}{pq}\right)\right)^j\\
&\ll_{D,j}& \left(\frac{p^2q^2K}{|x|c^2X}\right)^j
\ena
for any $j\geq 0$, recalling \eqref{assumption 2}. Thus for $j$ sufficiently large,
$\widehat{\mathcal{J}}(x,r,\ell;c,p)$ is negligibly small if
$|x|>N^{\varepsilon}p^2q^2K/(c^2X)$.
Therefore, we can write \eqref{expression} as
\bea\label{expression-1}
U\left(\frac{p^2q^2n}{DXc^2}\right)\mathcal{J}(n,r,\ell;c,p)
=\int_{|x|\leq N^{\varepsilon}p^2q^2K/(c^2X)}
\widehat{\mathcal{J}}(x,r,\ell;c,p)e(nx)\mathrm{d}x+O(N^{-2020}).
\eea
Moreover, we rewrite \eqref{J-hat} as
\bna\label{J-hat-1}
\widehat{\mathcal{J}}(x,r,\ell;c,p)&=&
\frac{DXc^2}{p^2q^2}\int_0^\infty\int_{\mathbb{R}}
\widetilde{V}_{p,\ell}(y)U\left(u\right)
e\left(e(G(y,u))\right)\mathrm{d}u
\mathrm{d}y
\ena
where
$$
G(y,u)=-\frac{t}{2\pi}\log y
-\frac{rNy}{[c,q]}+\frac{2\sqrt{NX\ell yu}}{pq}
-\frac{DXc^2u}{p^2q^2}x.
$$
Calculating the partial derivatives, one has
\bna
\begin{split}
\frac{\partial^2 G(y,u)}{\partial y^2} =& \,\frac{t}{2\pi y^2}-
\frac{\sqrt{NX\ell u}}{2pqy^{3/2}}\asymp \max\{t,K\}=t,\\
\frac{\partial^2 G(y,u)}{\partial u^2}
=&-\frac{\sqrt{NX\ell y}}{2pqu^{3/2}}\asymp K,
\end{split}\ena
and
$$\frac{\partial^2 G(y,u)}{\partial y^2}\cdot \frac{\partial^2 G(y,u)}{\partial u^2}
-\left(\frac{\partial^2 G(y,u) }{\partial y\partial u}\right)^2\gg tK.$$
Hence by applying the second derivative test in Lemma \ref{lem: 2nd derivative test, dim 2}
with $\rho_1=t$, $\rho_2=K$ and $\text{Var}=1$, we obtain
\bea\label{two dim}
\widehat{\mathcal{J}}(x,r,\ell;c,p)\ll_{D} \frac{Xc^2}{p^2q^2t^{1/2}K^{1/2}}.
\eea

\begin{lemma}\label{S-c-p-1}
We have
\bea\label{S-c-p-1}
\mathscr{S}_{c=p}(N)\ll \frac{N^{1/2+\varepsilon}L^{1/2}t^{1/2}}{q^{1/2}}.
\eea
\end{lemma}
\begin{proof}
Putting \eqref{expression-1} into \eqref{S-c-p} and using \eqref{two dim}, we have
\bna
\mathscr{S}_{c=p}(N)&=&\frac{1}{\eta_g(D)D^{1/2}}
\frac{N^{2-it}Lg_\chi}{q^2P^2K^{3/2}L^{\star}P^{\star}}\sum_{\ell\in \mathcal{L}}
\overline{\lambda_g(\ell)}\sum_{p\in\mathcal{P}}
\overline{\xi_{D}}(-p)\chi(p)
\sum_{0\neq  |r|\ll\frac{Pqt}{N^{1-\varepsilon}}\atop (r,p)=1}
\overline{\chi}(r)
\nonumber\\&\times&
\int_{|x|\leq N^{\varepsilon}q^2K/X}
\widehat{\mathcal{J}}(x,r,\ell;p,p)
\sum_{n\ll X/q^2}\lambda_{g^*}(n)
e\left(\frac{-\overline{Dr}\ell q n}{p}+nx\right)\mathrm{d}x+O(N^{-2020})\\
&\ll&\frac{N^{2+\varepsilon}L}{q^{3/2}P^2K^{3/2}} \frac{Pqt}{N}
\frac{q^2K}{X} \frac{X}{q^2t^{1/2}K^{1/2}}\frac{X^{1/2}}{q}\\
&\ll&\frac{N^{1/2+\varepsilon}L^{1/2}t^{1/2}}{q^{1/2}}.
\ena
Here we have used the bound (see \cite[Theorem 5.3]{Iwaniec} and \cite[Theorem 8.1]{Iwaniec1})
\bna
\sum_{n\leq X}\lambda_f(n)e(n \alpha)\ll_{f} X^{1/2}\log (2X),
\ena
which holds uniformly in $\alpha\in \mathbb{R}$.

\end{proof}

\subsection{Conclusion}

By \eqref{first estimate}, \eqref{reduction}, \eqref{S1-expression-4},
\eqref{S-sharp-1}, \eqref{S-c-q-1} and \eqref{S-c-p-1}, we conclude that
\bna
L\left(\frac{1}{2}+it, g\otimes\chi\right)\ll_{g,\varepsilon}
(qt)^{\varepsilon} q^{3/8}t^{1/3}.
\ena
This completes the proof of Theorem \ref{hybrid bound}.

\appendix
\section{Estimates for exponential integrals}
Let
\begin{equation*}
I = \int_{\mathbb{R}} w(y) e^{i \varrho(y)} dy.
\end{equation*}
Firstly, we have the following estimates for exponential integrals
(see \cite[Lemma 8.1]{BKY}  and \cite[Lemma A.1]{AHLQ}).
	
\begin{lemma}\label{lem: upper bound}
Let $w(x)$ be a smooth function    supported on $[ a, b]$ and
$\varrho(x)$ be a real smooth function on  $[a, b]$. Suppose that there
are parameters $Q, U,   Y, Z,  R > 0$ such that
\begin{align*}
\varrho^{(i)} (x) \ll_i Y / Q^{i}, \qquad w^{(j)} (x) \ll_{j } Z / U^{j},
\end{align*}
for $i \geqslant 2$ and $j \geqslant 0$, and
\begin{align*}
| \varrho' (x) | \geqslant R.
\end{align*}
Then for any $A \geqslant 0$ we have
\begin{align*}
I \ll_{ A} (b - a)
Z \bigg( \frac {Y} {R^2Q^2} + \frac 1 {RQ} + \frac 1 {RU} \bigg)^A .
\end{align*}
\end{lemma}

We also need the one- and two-dimensional
second derivative tests  (see Lemma 5.1.3 in \cite{Hux},
and Lemma 4 in \cite{munshi2}).

\begin{lemma}\label{lem: derivative tests, dim 1}
Let $f(x)$ be a real smooth function on  $[a, b]$.
Let $w(x)$ be a real smooth function supported on $[a, b]$
and let $V$ be its total variation. If $f''(x)\geq\lambda> 0$ on $[a, b]$, then
\bna
\left|\int_a^b e(f(x))w(x)\mathrm{d}x\right| \leq \frac{4V}{\sqrt{\pi\lambda}}.
\ena
\end{lemma}
	
\begin{lemma}\label{lem: 2nd derivative test, dim 2}
Let $h(x, y)$ be a real smooth function on $[a, b]\times [c, d]$ with
$$\left|\partial^2 h / \partial x^2 \right| \gg \lambda > 0, \hskip 15pt  \left|\partial^2 h / \partial y^2 \right| \gg \rho > 0,
$$
\\
$$|\det h''|=\left|\partial^2 h / \partial x^2 \cdot \partial^2 h / \partial y^2-(\partial^2 h / \partial x \partial y)^2 \right| \gg \lambda\rho$$,\\
on the rectangle $[a, b] \times [c, d]$. Let $w(x, y)$ be a real smooth function supported on $[a, b] \times [c, d]$ and let
\bna
V = \int_a^b \int_c^d \left|\frac{\partial^2 w(x, y)}{\partial x \partial y}\right|\mathrm{d}x\mathrm{d}y   .
\ena
Then
\bna
\int_a^b \int_c^d e(h(x, y))w(x, y)\mathrm{d}x \mathrm{d}y
\ll\frac{V}{\sqrt{\lambda\rho}},
\ena
with an absolute implied constant.
\end{lemma}

\subsection*{Acknowledgments.}
The authors express their thanks to
Yongxiao Lin and Zhi Qi for many illuminating discussions and suggestions.
Qingfeng Sun is partially supported by National Natural Science Foundation
of China (Grant No. 11871306).

\end{document}